\documentclass[12pt,a4paper]{article}

\usepackage{amssymb}
\usepackage{amsfonts} 
\usepackage{amsmath}  
\usepackage{latexsym} 
\usepackage{a4wide}

\usepackage{url}

\newtheorem{theorem}{THEOREM}[section]
\newtheorem{lemma}[theorem]{LEMMA}
\newtheorem{corollary}[theorem]{COROLLARY}

\newtheorem{problem}[theorem]{PROBLEM}
\newtheorem{example}[theorem]{EXAMPLE}
\newtheorem{defin}[theorem]{DEFINITION}

\newtheorem{fact}[theorem]{FACT}

\newenvironment{definition}{\begin{defin}\rm}{\end{defin}}

\def\zilch{\vrule height0pt width0pt depth0pt}

\def\endproof{\ifvmode\else\unskip\fi\zilch%
\penalty9999 \hbox{}\nobreak\hfill\quad$\Box$\endtrivlist}

\def\bigop#1{\mathop{\mathchoice
{\vcenter{\hbox{\huge$#1$}}}
{\vcenter{\hbox{\large$#1$}}}
{\vcenter{\hbox{\large$#1$}}}
{\vcenter{\hbox{$\scriptstyle #1$}}}}
\displaylimits}

\def\Exists{\bigop{\exists}}
\def\Lbig{\c L^{\forall\bdf\ang n}_{\bo\bod\dit}}
\def\R{\mathbb{R}}
\def\To{\Rightarrow}
\def\Var{\mathsf{Var}}
\def\Vec#1#2{#1_0,\ldots{},#1_{#2}}
\def\Vecc#1#2{#1_0,\ldots{},#1_{#2-1}}
\def\ang#1{\langle #1\rangle}
\def\b #1{\overline{#1}}
\def\ba{boolean algebra}
\def\bdf{[\neq]}
\def\bo{\Box}
\def\bod{[d]}
\def\c #1{{\mathcal #1}}
\def\cl{\mathop{\rm cl}\nolimits}
\def\claim{\par\medskip\noindent{\bf Claim. }}
\def\clop{{\sf Clop}}
\def\ddf{\ang{\neq}}
\def\di{\Diamond}
\def\did{{\ang d}}
\def\didt{{\ang{dt}}}
\def\dims{dense-in-itself \ms}
\def\dimss{dense-in-themselves \ms s}
\def\dit{{\ang t}}
\def\dm{dimen\-sion}
\def\fo{first-order}
\def\int{\mathop{\rm int}} 
\newcommand{\ints}{\mbox{\(\mathbb Z\)}}
\def\kudsys{S4DT$_1$S}
\def\lad{\c L^{\forall}_{\bod}}
\def\lbc{\c L^{\ang n}_\bo}  
\def\ms{metric space}
\def\nhd{neighbourhood}
\def\pa{$\forall$}
\def\pd{pairwise disjoint}
\def\pfclaim{\par\noindent{\bf Proof of claim. }}
\newcommand{\rats}{\mbox{\(\mathbb Q\)}}
\def\sem #1{[\![#1]\!]}
\def\setB{\mathbb{B}}
\def\setC{\mathbb{C}}
\def\setG{\mathbb{G}}
\def\setI{\mathbb{I}}
\def\str{structure}
\def\t{\ang t}
\def\uf{ultrafilter}
\def\vec#1#2{#1_1,\ldots{},#1_{#2}}

\def\unarycomplement{{\sim}}

\title{Strong 
completeness of modal logics
over 0-dimensional metric spaces\footnote{AMS 2010 classification: Primary 03B45; Secondary 54E35. Key words: dense in itself; Cantor set; coderivative operator; universal modality; difference modality; graded modalities.  We thank the referees for very helpful reports, and the editor Wesley Holliday for handling the paper.}}

\author{Robert Goldblatt\thanks{School of Mathematics and Statistics,
Victoria University,
 Wellington, New Zealand.
\tt Rob.Goldblatt@msor.vuw.ac.nz, sms.vuw.ac.nz/$\sim$rob}
\ and Ian Hodkinson\thanks{Department of Computing,
Imperial College London, London, UK. 
\tt  i.hodkinson@imperial.ac.uk,
doc.ic.ac.uk/$\sim$imh/}}

\begin{document} 
\maketitle

\begin{abstract}
We prove strong completeness 
results for some modal logics with the universal modality, with respect to
their topological semantics over 0-\dm al \dimss.
We also use failure of compactness to show that, for some languages and 
spaces, no standard
modal deductive system is strongly complete.
\end{abstract}

%

\section{Introduction}\label{sec:intro}

Modal languages can be given  semantics in a metric or topological space,
by interpreting $\bo$  as the interior operator.
This `topological semantics'
predates Kripke semantics and has a distinguished history.
In a celebrated result,
\cite{McKiT44,mcki:theo48:short} showed that
the logic of an arbitrary separable \dims\ in this semantics is 
the modal logic S4, 
whose chief axioms are $\bo \varphi\to \varphi$ and $\bo\varphi\to\bo\bo\varphi$.
The separability assumption was removed by \cite{RS:mm}.

So we can say two things.
Fix any \dims\ $X$ and
any set $\Sigma\cup\{\varphi\}$ of 
modal formulas, and  write `$\vdash$' for S4-provability.
First, $\vdash$ is \emph{sound over $X$:}
if $\Sigma\vdash\varphi$ then $\varphi$ is a semantic consequence of $\Sigma$
over $X$.
Second, $\vdash$ is \emph{complete over $X$:}
if $\varphi$ is a semantic consequence of $\Sigma$
over $X$, \emph{and $\Sigma$ is finite,} then $\Sigma\vdash\varphi$.

We say that a modal deductive system $\vdash$
is \emph{strongly complete over $X$}
if the second statement above holds for arbitrary --- even infinite --- sets $\Sigma$
of formulas.

\subsection{Some history}
Although
McKinsey and Tarski's result has been well known for a long time,
the study of strong completeness 
for modal languages in topological semantics 
seems to have begun only quite recently.
Gerhardt \cite[theorem 3.8]{gerhardt04}
proved that S4 is
strongly complete over the metric space $\rats$ of the rational numbers.
(He proved further results, in stronger  languages,
that imply our theorem~\ref{thm:stro compl noncompact} below for this
particular  space.)
The field opened out when 
\cite{Kremer2010:stro}
proved that S4 is strongly complete
over every \dims, thereby strengthening McKinsey and Tarski's  theorem.

In appendix I of \cite{McKiT44},
the authors suggested studying the more expressive 
`coderivative' operator $\bod$.
In the modal language incorporating this operator,
different \dimss\ have different logics and can need different treatment.
For this language and some
stronger ones incorporating the modal mu-calculus or the equivalent 
`tangle' operators, soundness and
strong completeness 
were shown by \cite{goldhod:tang17:apal}
for some deductive systems over some \dimss,
and by~\cite{gold:tang16} for other deductive systems over all
0-\dm al \dimss. More details will be given in \S\ref{ss:stro compl}

None of these languages include the universal modality \pa.
Indeed, in the presence of \pa, strong completeness cannot always be achieved.
No modal deductive system
for the language with $\bo$ and \pa\
is sound and strongly complete
over any compact locally connected \dims\ \cite[corollary 9.5]{goldhod:tang17:apal}.

\subsection{The work of this paper}\label{ss:intro of here}
Not covered by the last-mentioned result are the many
 \dimss\ that are not compact and locally connected.
For example, \emph{0-\dm al dense-in-them\-selves metric spaces} are almost never compact
(the only exception is the Cantor set) and never locally connected.
\emph{So in this paper, we study
strong completeness for 0-\dm al \dimss\ in languages able to express \pa.}
Sound and complete deductive systems for
these spaces in languages with \pa\ were given by~\cite{gold:tang16},
and for languages with the even more powerful `difference operator' $\bdf$
by~\cite{Kud06:topdiff}.
In this paper, we ask
whether the systems are
strongly complete.

The answer depends on both the language and the space,
making for an interesting variety as well as some novel techniques.
Our main conclusions are outlined in table~\ref{tab0}.

\begin{table}[h]
\begin{center}
\def\myskip{\hskip1em\null}
\begin{tabular}{ccc}
\hline
&
$\forall$
&
$[\neq]$
\\
\hline
$\bo$\myskip
&\myskip yes for all\myskip&\myskip yes for Cantor set; open for others
\\
$\bod$\myskip&\myskip  no for Cantor set; yes for others
\myskip&\myskip no for Cantor set; open for others
\\
\hline
\end{tabular}
\caption{For which 0-\dm al \dimss\ do we have strong completeness?}
\label{tab0}
\end{center}
\end{table}

In more detail,
let $X$ be a 0-\dm al \dims. 
\begin{enumerate}
\item In the language comprising $\forall$ and $\bo$,
the system S4U is strongly complete over 
$X$ (corollary~\ref{cor:S4U all 0dim}).

\item If $X$ is the Cantor set,
then in the language comprising $\bdf$ and $\bo$,
the system \kudsys\ is strongly complete
over $X$
(corollary~\ref{cor:kudinov}).

\item\label{it:Krem} If $X$ is not homeomorphic to the Cantor set, then in the language comprising $\forall$ and $\bod$,
the system KD4U is strongly complete
over $X$
(corollary~\ref{cor:str compl noncomp}).
\end{enumerate}
We will not need details of these systems,
but briefly, S4U comprises
the basic modal K axioms for $\bo$ and \pa,  the S4 axioms $\bo \varphi\to\varphi$ and $\bo\varphi\to\bo\bo\varphi$, and the U axioms $\forall\varphi\to\varphi$,
 $\varphi\to\forall\exists\varphi$, $\forall\varphi\to\forall\forall\varphi$,
and $\forall\varphi\to\bo\varphi$.
In KD4U, $\bo \varphi\to\varphi$ is replaced by the D axiom
$\di\top$ (and $\bo$ by $\bod$ throughout). 
The inference rules are modus ponens and universal generalisation.
The axioms of \kudsys\
boil down to the S4 axioms for $\bo$, the  K axioms for $[\neq]$,
$p\to[\neq]\ang{\neq}p$, $\forall p\to\bo p\wedge[\neq][\neq] p$,
and $[\neq] p\to\di p\wedge[\neq]\bo p$, where $\forall\varphi=\varphi\wedge[\neq]\varphi$; the rules are modus ponens, universal generalisation, and substitution.
Full definitions can be found
in, e.g.,~\cite[\S8.1]{goldhod:tang17:apal} and \cite{goldhod:tang17:sl},
and \cite[\S2]{Kud06:topdiff} for \kudsys\null.

To prove these results,
we will use  completeness theorems from \cite{gold:tang16}
and \cite{Kud06:topdiff}.
We lift them to strong completeness
by methods similar to those of
\cite{Kremer2010:stro} for non-compact spaces,
and  \fo\ compactness for the Cantor set.

Limitative results will also be given:
\begin{enumerate}
\setcounter{enumi}{3}
\item Let $X$ be a  \dims.
In any language able to express $\forall$ and the tangle operators (or the mu-calculus),
no modal deductive system is sound and strongly complete over $X$
(corollary~\ref{cor:tangle out}).

\item \label{it:compact T1}
Let $X$ be an infinite compact 
T1 topological space.
In any language able to express $\forall$ and $\bod$, no modal deductive system is
sound and strongly complete over $X$ (theorem~\ref{thm:str compl fails cantor [d],A}).

\end{enumerate}
One striking consequence is that
for the language comprising $\forall$ and $\bod$,
KD4U is sound and complete over 
every 0-\dm al \dims\ $X$
(by the discussion following \cite[theorem 8.4]{gold:tang16}), but 
by (\ref{it:Krem}) and (\ref{it:compact T1}),
 it is \emph{strongly} complete only when $X$ is not compact.
Over the Cantor set, no orthodox modal deductive system
for this language is strongly complete.

\section{Basic definitions}\label{sec:basics 1}

In this section, we give the main definitions and some notation.
We begin with some stock items.
We will use 
\ba s sometimes, and \uf s many times, and we refer the reader to, e.g., \cite{GiHa:ba}
for information.
Let $\c B=(B,+,-,0,1)$ be a \ba.
As usual, for elements 
$a,b\in B$
we write $a\leq b$ iff $a+b=b$, and $a\cdot b=-(-a+-b)$.
An \emph{atom} of $\c B$ is  a $\leq$-minimal nonzero element,
and $\c B$ is said to be \emph{atomless} if it has no atoms.
An \emph{\uf\ of $\c B$} is a subset $D\subseteq B$
such that for every $a,b\in B$ we have
$b\geq a\in D\To b\in D$,
$a,b\in D\To a\cdot b\in D$, and
$a\in D\iff -a\notin D$.
We say that $D$ is \emph{principal} if it contains an atom,
and \emph{non-principal} if not.

We denote the first infinite ordinal by $\omega$.
It is also a cardinal. 
For a set $S$,
we write $\wp(S)$ for its power set (set of subsets),
and $|S|$ for its cardinality.
We say that $S$ is \emph{countable}
if $|S|\leq\omega$, and \emph{countably infinite} if $|S|=\omega$.
An \emph{\uf\ on $S$} is an
\uf\ of the \ba\  $(\wp(S),\cup,\unarycomplement,\emptyset,S)$,
where $\unarycomplement$ denotes the unary complement operation
(we call such algebras, and subalgebras of them, \emph{boolean set algebras}).
The principal \uf s on $S$ are those of the form
$\{T\subseteq S:s\in T\}$ for $s\in S$.

\subsection{Kripke frames}\label{ss:kripke frames}

A \emph{(Kripke) frame} is a pair $\c F=(W,R)$,
where $W$ is a non-empty set of `worlds'
and $R$ is a binary relation on $W$.
For $w\in W$, we write $R(w)$ for $\{v\in W:R(w,v)\}$.
We say that $\c F$ is \emph{countable} if $W$ is countable,
\emph{serial} if $R(w)\neq\emptyset$ for every $w\in W$,
and \emph{transitive} if $R$ is transitive.

For frames $\c F=(W,R)$ and $\c F'=(W',R')$,
a \emph{p-morphism from $\c F$ to $\c F'$} is a map 
$f:W\to W'$
such that $f(R(w))=R'(f(w))$ for every $w\in W$.
See standard modal logic texts such as \cite{BRV:ml}
and \cite{ChagZak:ml} for information about p-morphisms.

\subsection{Topological spaces}\label{ss:top spaces}

We will assume some familiarity with topology,
but we give a rundown of the
main definitions and notation used later.
Other topological terms that we use occasionally,
and vastly more  information, can be found in 
topology texts such as \cite{Engelking89} and \cite{Will70}
(these two will be our main references).

A \emph{topological space} is a pair $(X,\tau)$,
where $X$ is a non-empty set and $\tau\subseteq\wp(X)$ satisfies:
\begin{enumerate}

\item if $\c S\subseteq\tau$ then $\bigcup\c S\in\tau$,

\item if $\c S\subseteq\tau$ is finite then $\bigcap\c S\in\tau$,
on the understanding that $\bigcap\emptyset=X$.

\end{enumerate}
So $\tau$ is a set of subsets of $X$ closed under unions and finite intersections.
Such a set is called a \emph{topology on $X$.}
By taking $\c S=\emptyset$,
it follows that $\emptyset,X\in\tau$.
The elements of $\tau$ are called \emph{open subsets} of $X$, or just \emph{open sets.}
An \emph{open \nhd} of a point $x\in X$ is an open set containing $x$.
A subset $C\subseteq X$ is called \emph{closed} if $X\setminus C$ is open,
and \emph{clopen} if it is both closed and open.
The set of closed subsets of $X$ is closed under intersections and finite unions.
Writing $\clop(X)$ for the set of clopen subsets of $X$,
$(\clop(X),\cup,\unarycomplement,\emptyset,X)$ is a boolean set algebra.
If $O$ is open and $C$ closed then  $O\setminus C$ is open and
$C\setminus O$ is closed.

We use the  signs $\int$, $\cl$, $\did$ to denote the \emph{interior,}
\emph{closure,} and \emph{derivative} operators, respectively.
So for $S\subseteq X$, 
\begin{itemize}
\item $\int S=\bigcup\{O\in\tau:O\subseteq S\}$ --- the largest
open set contained in $S$,

\item   $\cl S=\bigcap\{C\subseteq X:C$ closed, $S\subseteq C\}$
--- the smallest closed set containing $S$;
we have $\cl S=\{x\in X:S\cap O\neq\emptyset
\mbox{ for every open \nhd\ } O\mbox{ of } x\}$,

\item $\did S=\{x\in X:S\cap O\setminus\{x\}\neq\emptyset
\mbox{ for every open \nhd\ } O\mbox{ of } x\}$.

\end{itemize}
For all subsets $A,B$ of $X$, we have 
\[
\begin{array}{rcl}
\cl(A\cup B)&=&\cl A\cup\cl B,
\\
\did(A\cup B)&=&\did A\cup\did B, 
\\
\int(A\cap B)&=&\int A\cap\int B.
\end{array}
\]
That is, \emph{closure and $\did$ are additive and interior is multiplicative.}
It follows that they are all \emph{monotonic:}
if $A\subseteq B$ then $\cl A\subseteq\cl B$,
$\did A\subseteq\did B$, and $\int A\subseteq\int B$.

Fix a topological space $(X,\tau)$.
A \emph{subspace} of $(X,\tau)$ is 
a topological space of the form $(Y,\tau_Y)$ where $\emptyset\neq Y\subseteq X$
and $\tau_Y=\{O\cap Y:O\in\tau\}$.

For a set $\tau_0\subseteq\wp(X)$,
the closure $\tau$ of $\tau_0$ under arbitrary unions and finite intersections
is a topology on $X$, called the topology \emph{generated} by $\tau_0$.
A \emph{base} for (the topology $\tau$ on) $(X,\tau)$
is a set $\tau_0\subseteq\tau$
such that $\tau=\{\bigcup \c S:\c S\subseteq\tau_0\}$.

An \emph{open cover} of $(X,\tau)$
is a subset $\c S\subseteq\tau$
 with $\bigcup\c S=X$.
We say then that $\c S$ is \emph{locally finite} if
every $x\in X$ has an open \nhd\ disjoint from  all  but finitely many
sets in~$\c S$.
An open cover $\c S'$ of $(X,\tau)$ is a 
\emph{subcover} of $\c S$ if $\c S'\subseteq\c S$,
and a \emph{refinement} of $\c S$ 
if
for every $S'\in\c S'$ there is $S\in\c S$ with $S'\subseteq S$.

The following assorted topological properties are well known and much studied.
We say that $(X,\tau)$ is 
\emph{dense in itself} if
no singleton subset of $X$ is open; 
\emph{T1} if  every singleton subset of $X$ is closed;
\emph{T2} if  every two distinct points of $X$ have disjoint open \nhd s;
 \emph{0-\dm al} if
it is T1 and has a base consisting of clopen sets;
\emph{separable} if $X$ has a countable subset $D$ with $X=\cl D$;
 \emph{Lindel\"of}
if every open cover of $X$ has a countable subcover;
 \emph{compact}
if every open cover of $X$ has a finite subcover;
and \emph{paracompact} if it is T2 and  every open cover 
of $(X,\tau)$ refines to a locally
finite  open cover of $(X,\tau)$.
(Not everyone requires that
0-\dm al spaces be T1 or that
paracompact spaces be T2, and some
writers add extra conditions such as T2 or regularity to the definitions
of compact and Lindel\"of. 
The spaces involved in this paper meet all these conditions.)
Easily, T2 implies T1.

We follow standard practice and identify (notationally) the space $(X,\tau)$
with $X$.

\subsection{Metric spaces}\label{ss:metric spaces}

A \emph{metric space} is a pair $(X,d)$,
where $X$ is a 
non-empty
set
and $d:X\times X\to\R$ is a `distance function'
(having nothing to do with the operator $\did$ above) satisfying, for all $x,y,z\in X$,
\begin{enumerate}
\item $d(x,y)=d(y,x)\geq0$,

\item $d(x,y)=0$ iff $x=y$,

\item $d(x,z)\leq d(x,y)+d(y,z)$ (the `triangle inequality').
\end{enumerate}
Examples of metric spaces abound and include the real numbers
$\R$ with the standard distance function $d(x,y)=|x-y|$,
$\R^n$ with Pythagorean distance, etc.
As usual, we  often identify (notationally) $(X,d)$ with $X$.

Let $(X,d)$ be a metric space.
A \emph{subspace} of $(X,d)$ is a metric space of the form
$(Y,d\restriction Y\times Y)$, for non-empty $Y\subseteq X$.
For $x\in X$ define 
$
d(x,Y)=\inf\{d(x,y):y\in Y\}.
$
We leave $d(x,\emptyset)$ undefined.
For a real number $\varepsilon>0$,
we let $N_\varepsilon(x)$ denote the  `open ball' 
$\{y\in X:d(x,y)<\varepsilon\}$, and for
$S\subseteq X$ we put $N_\varepsilon(S)=\bigcup\{N_\varepsilon(x):x\in S\}$.
A metric space $(X,d)$ gives rise to a topological space $(X,\tau_d)$ 
in which
a subset $O\subseteq X$ is declared to be open (i.e., in $\tau_d$)
iff for every $x\in O$, there is some $\varepsilon>0$ such that
$N_\varepsilon(x)\subseteq O$.
In other words, the open sets are the unions of open balls.
We will say that a metric space has a given topological property
(such as being dense in itself) if its associated topological space has the property.
For example, it is known  that every \ms\ is T2 (easy), and  paracompact
(\cite{S48}).

\subsection{Modal languages}

We fix a countably infinite set $\Var$ of \emph{propositional variables,} or \emph{atoms.}
We will be considering a number of modal languages.
The biggest of them is
denoted by 
$\Lbig$, 
which is a set of formulas defined as follows:
\begin{enumerate}
\item each $p\in\Var$ is a formula (of $\Lbig$),

\item $\top$ is a formula,

\item if $\varphi,\psi$ are formulas then so are
$\neg\varphi$, $(\varphi\wedge\psi)$,
$\bo\varphi$,  $\bod\varphi$,  $\forall\varphi$,
$\bdf\varphi$,
and $\ang n\varphi$ for each $n<\omega$,

\item if $\Delta$ is a non-empty finite set of formulas then
$\dit\Delta$  is a formula.

\end{enumerate}
We use standard abbreviations:
$\bot$ denotes $\neg\top$,
$(\varphi\vee\psi)$
denotes $\neg(\neg\varphi\wedge\neg\psi)$,
$(\varphi\to\psi)$ denotes $\neg(\varphi\wedge\neg\psi)$,
$(\varphi\leftrightarrow\psi)$ denotes
$((\varphi\to\psi)\wedge(\psi\to\varphi))$,
$\di\varphi$ denotes $\neg\bo\neg\varphi$,
$\did\varphi$ denotes $\neg\bod\neg\varphi$,
$\exists\varphi$ denotes $\neg\forall\neg\varphi$,
and 
$\ddf\varphi$ denotes $\neg\bdf\neg\varphi$.
Parentheses will be omitted where possible, by the usual methods.
For a non-empty finite set $\Delta=\{\vec\delta n\}$ of formulas,
we let $\bigwedge\Delta$ denote $\delta_1\wedge\ldots\wedge\delta_n$
and $\bigvee\Delta$ denote $\delta_1\vee\ldots\vee\delta_n$
(the order and bracketing of the conjuncts and disjuncts will always be immaterial). 
We set $\bigwedge\emptyset=\top$ and $\bigvee\emptyset=\bot$.

The connective $\bod$ is called
the \emph{coderivative operator,}
and the connective $\dit$ is called the \emph{tangle connective}
or  \emph{tangled closure operator.}  A more powerful tangle connective
$\didt$ can also be considered 
(see, e.g., \cite{gold:tang16,goldhod:tang17:apal})
but we will not need it here.
The connectives $\forall$ and $\bdf$ are called the \emph{universal} and 
\emph{difference} modalities, respectively, and
the connectives $\ang n$ are sometimes called the \emph{counting} or \emph{graded} modalities.

We will be using various \emph{sublanguages} of $\Lbig$,
and they will be denoted in the obvious way by omitting prohibited operators
from the notation.
So for example, $\c L^{\forall}_{\bo}$
denotes the set of all $\Lbig$-formulas
that do not involve $\bod,$  $\dit$, $\bdf$, or any $\ang n$.

\subsection{Kripke semantics}\label{ss:Kripke sem}

An \emph{assignment} or \emph{valuation} into a frame $\c F=(W,R)$ is
a map $h:\Var\to\wp(W)$.
A \emph{Kripke model} is a triple $\c M=(W,R,h)$, where $(W,R)$ is a frame
and $h$ an assignment into it.
The \emph{frame of $\c M$} is $(W,R)$.

For every Kripke model $\c M=(W,R,h)$ and every world $w\in W$, we define the notion 
$\c M,w\models\varphi$ of a formula $\varphi$ of
$\Lbig$ being \emph{true at $w$ in $\c M$.}
The definition is by induction on $\varphi$,
as follows:
\begin{enumerate}
\item $\c M,w\models p$ iff $w\in h(p)$, for $p\in\Var$.

\item $\c M,w\models\top$.

\item $\c M,w\models\neg\varphi$ iff $\c M,w\not\models\varphi$.

\item $\c M,w\models\varphi\wedge\psi$ iff $\c M,w\models\varphi$ and $\c M,w\models\psi$.

\item $\c M,w\models\bo\varphi$ iff $\c M,v\models\varphi$ for every
$v\in R(w)$.

\item The truth condition for 
$\bod\varphi$ is exactly the same as for $\bo\varphi$.

\item\label{item: sem tangle kripke} $\c M,w\models\dit\Delta$ iff there
are worlds $w=w_0,w_1,\ldots\in W$ with $R(w_n,w_{n+1})$ for each $n<\omega$
and such that for each $\delta\in\Delta$ there are infinitely many $n<\omega$
with $\c M,w_n\models\delta$.

\item $\c M,w\models\forall\varphi$ iff $\c M,v\models\varphi$
for every $v\in W$.

\item $\c M,w\models\bdf\varphi$ iff $\c M,v\models\varphi$
for every $v\in W\setminus\{w\}$.

\item $\c M,w\models\ang n\varphi$ iff 
$|\{v\in W:\c M,v\models\varphi\}|>n$.

\end{enumerate}
For a set $\Gamma$ of formulas,
we write $\c M,w\models\Gamma$
if $\c M,w\models\gamma$ for every $\gamma\in\Gamma$.

\subsection{Topological semantics}\label{ss:topsem}
Given a topological space $X$, 
an \emph{assignment} (or \emph{valuation}) into $X$ is 
a map $h:\Var\to\wp(X)$.
A \emph{topological model} is a pair $(X,h)$, where
$X$ is a topological space and $h$ an assignment into~$X$.
For every topological model $(X,h)$ and every point $x\in X$, we define 
$(X,h),x\models\varphi$, for a
$\Lbig$-formula $\varphi$, by induction on $\varphi$:
\begin{enumerate}
\item $(X,h),x\models p$ iff $x\in h(p)$, for $p\in\Var$.

\item $(X,h),x\models\top$.

\item $(X,h),x\models\neg\varphi$ iff $(X,h),x\not\models\varphi$.

\item $(X,h),x\models\varphi\wedge\psi$ iff $(X,h),x\models\varphi$ and $(X,h),x\models\psi$.

\item $(X,h),x\models\bo\varphi$ iff 
there is an open \nhd\ $O$ of $x$ with $(X,h),y\models\varphi$ for every
$y\in O$.

\item $(X,h),x\models\bod\varphi$ iff 
there is an open \nhd\ $O$ of $x$ with $(X,h),y\models\varphi$ for every
$y\in O\setminus\{x\}$.  

\item\label{topsem clause 8} For a non-empty 
finite set $\Delta$ of formulas for which we have 
inductively defined semantics,
write $\sem\delta=\{x\in X:(X,h),x\models\delta\}$
for each $\delta\in\Delta$.
Then:

 $(X,h),x\models\dit\Delta$ iff 
there is some $S\subseteq X$ such that
$x\in S\subseteq\bigcap_{\delta\in\Delta}\cl(\sem\delta\cap S)$.

\item $(X,h),x\models\forall\varphi$ iff 
$(X,h),y\models\varphi$ for every $y\in X$.

\item $(X,h),x\models\bdf\varphi$ iff 
$(X,h),y\models\varphi$ for every $y\in X\setminus\{x\}$.

\item $(X,h),x\models\ang n\varphi$ iff  $|\{y\in X:(X,h),y\models\varphi\}|>n$.

\end{enumerate}
Writing $\sem\varphi=\{x\in X:(X,h),x\models\varphi\}$,
we have $\sem{\bo\varphi}=\int(\sem\varphi)$,
$\sem{\di\varphi}=\cl(\sem\varphi)$, and
$\sem{\did\varphi}=\did(\sem\varphi)$ for each $\varphi$.

As with Kripke semantics,
for a set $\Gamma$ of formulas
we write $(X,h),x\models\Gamma$
if $(X,h),x\models\gamma$ for every $\gamma\in\Gamma$.
We say that $\Gamma$ is \emph{satisfiable in $(X,h)$}
if $(X,h),x\models\Gamma$ for some $x\in X$; and 
\emph{satisfiable in $X$}
if it is satisfiable in $(X,h)$ for some assignment $h$ into $X$.
We say that $\Gamma$ is \emph{finitely satisfiable in $(X,h)$ (respectively, $X$)}
if every finite subset of $\Gamma$ is satisfiable in $(X,h)$ (respectively, $X$).
Of course, we say that a formula $\varphi$ is satisfiable in these ways
if $\{\varphi\}$ is so satisfiable.
We write $\Gamma\models_X\varphi$
if  $\Gamma\cup\{\neg\varphi\}$ is not satisfiable in $X$.
For a language $\c L\subseteq\Lbig$,
the \emph{$\c L$-logic of $X$} is the set 
$\{\varphi\in\c L:\emptyset\models_X\varphi\}$.

\subsection{Weaker, stronger, and equivalent languages}\label{ss:trans}

We say that formulas $\varphi,\psi$
are (topologically) \emph{equivalent}
if $(X,h),x\models\varphi\leftrightarrow\psi$ for
every topological model $(X,h)$ and every $x\in X$.
For languages $\c L,\c L'\subseteq\Lbig$,
we say that $\c L$ is \emph{weaker than} $\c L'$,
and $\c L'$ is \emph{stronger than} $\c L$,
if every formula of $\c L$ is equivalent to a formula of $\c L'$.
We say that $\c L$ is \emph{equivalent to} $\c L'$ if $\c L$ is
both weaker and stronger than $\c L'$,
and that $\c L$ is \emph{strictly weaker than} $\c L'$,
and $\c L'$ is \emph{strictly stronger than} $\c L$,
if $\c L$ is weaker but not stronger than $\c L'$.

Some operators of $\Lbig$ can express others.
Clearly, $\bo\varphi$ is 
(topologically) equivalent to $\varphi\wedge\bod\varphi$
and to
$\neg\dit\{\neg\varphi\}$,
and $\forall\varphi$ is equivalent to 
$\neg\ang0\neg\varphi$.
It follows for example that
$\c L^\forall_\bo$, $\c L^{\forall\ang n}_{\bo\bod}$
are weaker than $\c L^{\ang n}_{\bod}$, and in fact the first strictly so.

In the same vein, $\ddf\varphi$ is equivalent to
$(\neg\varphi\to\exists\varphi)\wedge(\varphi\to\ang1\varphi)$,
$\exists\varphi$ is equivalent to $\varphi\vee\ddf\varphi$,
and $\ang1\varphi$ is equivalent to $\exists(\varphi\wedge\ddf\varphi)$.
So we can exchange $\{\forall,\ang1\}$ with $\bdf$,
preserving language equivalence;
and the language $\c L^{\bdf}_\zeta$ is weaker than $\c L^{\ang n}_\zeta$,
for any $\zeta$.

\subsection{Strong completeness}\label{ss:stro compl}

This is the topic of the paper.
We assume familiarity,
e.g., from \cite{gold:tang16} and \cite[\S\S2.10, 2.12, 8.1]{goldhod:tang17:apal}, with 
(modal) deductive systems.
They are Hilbert systems containing, at least, all propositional tautologies
as axioms  and the modus ponens inference rule.
For such a system $\vdash$, a \emph{theorem} of $\vdash$ is a formula $\varphi$
that is provable in $\vdash$, in which case we write $\vdash\varphi$;
for a set $\Sigma$ of formulas,
we write $\Sigma\vdash\varphi$
if there is some finite $\Sigma_0\subseteq\Sigma$
such that $\vdash(\bigwedge\Sigma_0)\to\varphi$;
and $\Sigma$ is said to be ($\vdash$-)\emph{consistent}
if $\Sigma\not\vdash\bot$.
All deductive systems mentioned later in the paper are taken to be of this form.
For such systems, though not for all deductive systems in the world, consistency reduces to a property
of the set of theorems,
and $\Sigma$ is consistent iff each of its finite subsets is consistent.

A deductive system $\vdash$ for a language $\c L\subseteq\Lbig$
is said to be 
\emph{sound} over a topological space $X$ if for every $\c L$-formula $\varphi$, if $\vdash\varphi$ then $\emptyset\models_X\varphi$.
Equivalently, every finitely satisfiable (in $X$) set of $\c L$-formulas is $\vdash$-consistent.
We say that $\vdash$ is \emph{strongly complete} over 
$X$ if for every set $\Sigma\cup\{\varphi\}$ of $\c L$-formulas,
if $\Sigma\models_X\varphi$ then $\Sigma\vdash\varphi$,
and \emph{complete} over $X$ if this holds when $\Sigma$ is finite.
It follows that $\vdash$ is (strongly)
complete over $X$ iff every finite $\vdash$-consistent
set (respectively, every $\vdash$-consistent set) of formulas is satisfiable in $X$.
Recall that $\Var$ is countable,
so we are dealing
always with countable sets of formulas.

For many topological spaces and sublanguages of $\c L_{\bo\bod\dit}$,
strongly complete deductive systems
are known.
\begin{itemize}
\item  \cite{Kremer2010:stro} showed that for $\c L_\bo$, the system {S4} is strongly complete over every \dims.
(It had long been known 
from the work of \cite{McKiT44,mcki:theo48:short} that {S4} is 
sound and complete over every such space.)

\item In the language $\c L_{\bo\dit}$,
the system S4$t$ is sound and strongly complete
over every \dims\ \cite[theorem 9.3(1)]{goldhod:tang17:apal}.

\item In the language $\c L_{\bod}$,
the system KD4G$_1$
is strongly complete over every
\dims, and sound if the space has a property
called `G$_1$' \cite[theorem 9.2]{goldhod:tang17:apal}.

\item The same holds for the system {KD4G$_1t$}
in a language  expanding $\c L_{\bod}$ by the stronger tangle operator
$\didt$ already mentioned 
\cite[theorem 9.1]{goldhod:tang17:apal}.

\item In this latter language, the system KD4\emph{t} is sound and  strongly complete over every 0-\dm al \dims\
\cite[theorem 8.5]{gold:tang16}.

\end{itemize}

\subsection{Compactness}\label{ss:compactness}

For a language $\c L\subseteq\Lbig$ and a topological space $X$,
we say that $\c L$ is \emph{compact over $X$}
if every set of $\c L$-formulas that is finitely satisfiable
in $X$ is satisfiable in~$X$.
Do not confuse this with compactness of the space $X$.

Obviously, if $\c L$ is compact over $X$
then so is every sublanguage of $\c L$,
and every weaker  language.
For example, 
if $\c L^{\ang n}_{\bod}$ is compact over $X$
then so are $\c L^\forall_\bo$, $\c L^{\forall\ang n}_{\bo\bod}$, etc.

Compactness is tightly connected to strong completeness.
The following is well known and easy to prove.
\begin{fact}\label{fact:stro compl = compact}\rm
Let $\vdash$ be a 
deductive system for a language $\c L\subseteq\Lbig$,
and let $X$ be a topological space.
If $\vdash$ is complete over $X$ and $\c L$ is compact over $X$,
then $\vdash$ is strongly complete over $X$.
The converse holds if $\vdash$ is sound over $X$.
\end{fact}
So on the one hand,
where  a complete deductive system is known for a space,
compactness, if available, can be used to show that 
the system is 
actually strongly complete.
Soundness is not required.
This is how the  results of
\cite{gold:tang16,goldhod:tang17:apal} 
mentioned in \S\ref{ss:stro compl}
were proved.

On the other hand, failure of compactness kills any hope
of finding a sound and strongly complete  deductive system.
As we mentioned in \ref{sec:intro},
no deductive system 
for $\c L^\forall_\bo$ is sound and strongly complete
over a compact locally connected \dims\
\cite[corollary 9.5]{goldhod:tang17:apal},
and this was proved using failure of compactness.

This paper is about
strong completeness over 0-\dm al \dimss\ in languages able to express \pa.
Relevant sound and complete deductive systems
were given by \cite{Kud06:topdiff} and \cite{gold:tang16},
and we are therefore interested in determining
which sublanguages of $\Lbig$ are 
{compact} over which 0-\dm al \dimss.
The rest of the paper is devoted to this question,
and the answers are varied and interesting.

\section{Strong completeness with $\forall$ and tangle fails always}

The following is based on an example
in \cite[\S5]{goldhod:tang17:sl} using $\bo$. Here we use $\forall$ instead.

\begin{theorem}\label{thm:stro comp with forall, tangle}
Compactness
fails for the language $\c L^{\forall}_{\dit}$
over every \dims~$X$.
\end{theorem}

\begin{proof}
Since $\c L^{\forall}_{\dit}$ can express $\bo\varphi$,
via $\neg\dit\{\neg\varphi\}$, 
we can work in $\c L^{\forall}_{\bo\dit}$.
Fix pairwise distinct atoms $q,p_0,p_1,\ldots\in\Var$, and define
\[
\Sigma=\{\neg\t\{q,\neg q\},p_0,\forall(p_n
\to\di p_{n+1}),\forall(p_{2n}\to q),
\forall(p_{2n+1}\to\neg q):n<\omega\}.
\]
For each $n<\omega$,
the subset $\Sigma_n$ of formulas in $\Sigma$
using atoms $\Vec pn,q$ only
is true at 0 in 
the Kripke model
$M_n=(\{0,\ldots,n\},\leq,h)$,
with $h(p_i)=\{i\}$ for $i\leq n$,
and $h(q)=\{2i:i<\omega,\;2i\leq n\}$.
The frame of $M_n$ validates the axioms of the system
S4$t$.UC of \cite[\S8.1]{goldhod:tang17:apal},
so $\Sigma_n$ is S4$t$.UC-consistent.
Now by \cite[theorem 8.4(2)]{goldhod:tang17:apal},
S4$t$.UC is complete over every \dims, and $\Sigma_n$ is finite, so
$\Sigma_n$ is satisfiable in $X$.
It follows that $\Sigma$ is finitely satisfiable in~$X$.

Suppose for contradiction that $(X,h),x_0\models\Sigma$,
for some assignment $h$ and some $x_0\in X$.
Below, we write $x\models\varphi$ as short for $(X,h),x\models\varphi$.
Let 
\[
S=\bigcup\{h(p_n):n<\omega\}\subseteq X.
\]
We show that $S\subseteq\cl(S\cap h(q))\cap\cl(S\setminus h(q))$.
Let $x\in S$.
Pick $n<\omega$ such that $x\models p_n$.
Suppose that $n$ is even
(the case where it is odd is similar). 
Since $x_0\models\forall(p_{n}\to q)$,
we have $x\in S\cap h(q)$ already,
so certainly $x\in\cl(S\cap h(q))$.
Now  let $O$ be any open \nhd\ of $x$.
Since $x_0\models\forall(p_n\to\di p_{n+1})$,
and $x\models p_n$,
there is $y\in O$ with $y\models p_{n+1}$.
So $y\in S$, and also 
 $y\models\neg q$ as  $x_0\models\forall(p_{n+1}\to\neg q)$ because $n+1$ is odd.
As $O$ was arbitrary, $x\in\cl(S\setminus h(q))$.
As $x$ was arbitrary, 
$S\subseteq\cl(S\cap h(q))\cap\cl(S\setminus h(q))$ as required.

By semantics of tangle (\S\ref{ss:topsem}), every point in $S$ satisfies $\t\{q,\neg q\}$.
Since $x_0\in h(p_0)\subseteq S$, $x_0\models\t\{q,\neg q\}$, 
contradicting that $x_0\models\Sigma$.
\end{proof}

\noindent 
The proof  applies to any language able to 
 express \pa, $\bo$, and $\ang t\{\varphi,\neg\varphi\}$.
The following is immediate via fact~\ref{fact:stro compl = compact}.
\begin{corollary}\label{cor:tangle out}
Let $X$ be a \dims.
No  deductive system for $\c L^{\forall}_{\dit}$
or any stronger language is sound and strongly complete 
over $X$.
\end{corollary}

One such stronger language comprises
$\bo$, \pa\, and the modal mu-calculus
 \cite[lemma 4.2]{goldhod:tang17:apal}.

By corollary~\ref{cor:tangle out}, in the presence of \pa\ we can forget about tangle.

\section{Non-compact 0-\dm al spaces with \pa\ and $\bod$}

We now aim to show that $\c L^\forall_{\bod}$ is compact
over every non-compact 0-\dm al \dims.
This will have consequences for strong completeness in the languages
$\c L^\forall_{\bod}$ and $\c L^\forall_{\bo}$.

\subsection{Topology}
We will need some topology.
Fix a \dims\ $(X,d)$.

\begin{fact}\label{fact:from apal}\rm
First we quote some basic results, some of which are true much more generally. They are easy to prove.
\begin{enumerate}
\item\label{open inf} \cite[lemma 5.3]{goldhod:tang17:apal} Every non-empty open subset of $X$ is infinite.

\item If $S\subseteq X$ then $\int S\subseteq S\cap\did S$ and $\cl S=S\cup\did S$.

\item  $\did$ is additive: if $S,T\subseteq X$ then $\did(S\cup T)=\did S\cup\did T$
(as already mentioned).

\item \cite[lemma 5.1(2)]{goldhod:tang17:apal}  
If $N\subseteq X$ has empty interior and $O\subseteq X$ is open, then
$\cl(O\setminus N)=\cl O$.

\end{enumerate}
\end{fact}

The following will be useful.
For a real number $\varepsilon>0$, we say that a subset $S\subseteq X$ is \emph{$\varepsilon$-sparse} if 
$d(x,y)\geq \varepsilon$ for every distinct $x,y\in S$.
In that case, $\did S=\emptyset$.

\begin{lemma}\label{lem:simple Zn}
Let $\setG\subseteq X$ be open
and let $I$ be a countable index set.
Then there are \pd\ sets $\setI_i\subseteq\setG$ $(i\in I)$ such that
\begin{enumerate}
\item $\did\setI_i=\cl\setG\setminus\setG$ for every $i\in I$,

\item $\setG\cap\did\bigcup_{i\in I}\setI_i =\emptyset$.
\end{enumerate}

\end{lemma}
Without part 2, this follows from
\cite[theorem 6.1]{goldhod:tang17:apal},
and part 2 can be extracted from the proof of that theorem.
But the lemma is fairly quick to prove, so we prove it here.

\begin{proof}
Write $B=\cl\setG\setminus\setG$.
If $B=\emptyset$, we can take $\setI_i=\emptyset$ for each $i\in I$.
We are done.

Assume now that $B\neq\emptyset$.
Define $\varepsilon_n=1/2^n$ for each $n<\omega$.
We define \pd\ subsets $Z_n\subseteq\setG$ ($n<\omega$),
with $\did Z_n=\emptyset$, by induction as follows.
Let $n<\omega$ and assume inductively that
$Z_m$ has been defined for each $m<n$.
Let 
\[
O_n=\setG\cap N_{\varepsilon_n}(B)\setminus\bigcup_{m<n}Z_m.
\]
Using Zorn's lemma,  
choose $Z_n$ to 
be a maximal $\varepsilon_n$-sparse subset of $O_n$.
As we said, $\did Z_n=\emptyset$, and plainly $Z_n\subseteq\setG$.
This completes the definition of the \pd\ $Z_n$.

We first show that 
\begin{equation}\label{e:Z deriv empty}
\setG\cap\did\bigcup_{n<\omega}Z_n=\emptyset.
\end{equation}
Let $x\in\setG$ be arbitrary, and choose $n<\omega$ so large that 
$N_{2\varepsilon_n}(x)\subseteq\setG$.
Consequently, $d(x,B)\geq 2\varepsilon_n$.
Now for each $m\geq n$ we have $Z_m\subseteq O_m\subseteq N_{\varepsilon_n}(B)$.
If there is $z\in N_{\varepsilon_n}(x)\cap Z_m$,
then $d(x,B)\leq d(x,z)+d(z,B)<\varepsilon_n+\varepsilon_n=2\varepsilon_n$, a contradiction.
So $N_{\varepsilon_n}(x)\cap \bigcup_{m\geq n}Z_m=\emptyset$,
and $x\notin\did\bigcup_{m\geq n}Z_m$.
By fact~\ref{fact:from apal}, $\did\bigcup_{m<n}Z_m=\bigcup_{m<n}\did Z_m=\emptyset$,
so $x\notin\did\bigcup_{m< n}Z_m$ as well.
Hence, $x\notin\did\bigcup_{m< n}Z_m\cup\did\bigcup_{m\geq n}Z_m=\did\bigcup_{m<\omega}Z_m$, proving \eqref{e:Z deriv empty}.

Now let $J\subseteq\omega$ be infinite;
we show that
\begin{equation}\label{e:Z deriv B}
\did\bigcup_{n\in J}Z_n = B.
\end{equation} 
Write $Z=\bigcup_{n\in J}Z_n$.
Certainly, since $Z\subseteq\setG$ we have $\did Z\subseteq \cl\setG$.
By \eqref{e:Z deriv empty} and monotonicity of $\did$, $\setG\cap\did Z=\emptyset$, so $\did Z\subseteq B$.

For the converse, let $b\in B$
and let $\varepsilon>0$ be given.
We will show that $Z\cap N_{\varepsilon}(b)\neq\emptyset$.

Choose $n\in J$ so large that $2\varepsilon_n\leq\varepsilon$.
By fact~\ref{fact:from apal},
$\int\bigcup_{m<n}Z_m\subseteq\did\bigcup_{m<n}Z_m=\emptyset$.
So $\bigcup_{m<n}Z_m$ has empty interior.
By fact~\ref{fact:from apal} again, 
$\cl\setG=\cl(\setG\setminus\bigcup_{m<n}Z_m)$.

Now $b\in\cl\setG$.
So there is $x\in N_{\varepsilon_n}(b)\cap\setG\setminus\bigcup_{m<n}Z_m\subseteq O_n$.
If $Z_n\cap N_{\varepsilon}(b)=\emptyset$,
then
for every $z\in Z_n$ we have 
$d(x,z)\geq d(b,z)-d(b,x)>\varepsilon-\varepsilon_n\geq\varepsilon_n$, 
so $x$ could be added to $Z_n$, contradicting its maximality.
Hence, $Z\cap N_{\varepsilon}(b)\neq\emptyset$, as required.

This holds for every $\varepsilon>0$,
and hence $b\in\cl Z=Z\cup\did Z$ (fact~\ref{fact:from apal}).
Since $Z\subseteq\setG$, we have $b\notin Z$, so $b\in\did Z$.
As $b\in B$ was arbitrary, we obtain $B\subseteq\did Z$, so proving \eqref{e:Z deriv B}.

Now to prove the lemma, simply partition $\omega$
into infinite sets $J_i$ $(i\in I)$
and define $\setI_i=\bigcup_{n\in J_i}Z_n$.
\end{proof}

From now on, assume further that $X$ is 0-\dm al.

\begin{lemma}\label{lem:aiml mainly}
Let $\setG\subseteq X$ be open, and suppose that $Z\subseteq \setG$
and $\setG\cap\did Z=\emptyset$.
Then there is a family $(K(T):T\subseteq Z)$ of subsets
of $\setG$ such that for each $T\subseteq Z$:
\begin{enumerate}
\item $T\subseteq K(T)\subseteq\setG$,

\item if $\c U\subseteq \wp(Z)$ then
$K(\bigcup\c U)=\bigcup_{U\in\c U}K(U)$, and hence $K(\emptyset)=\emptyset$,

\item  if $U\subseteq Z$ and $T\cap U=\emptyset$ 
then $K(T)\cap K(U)=\emptyset$,

\item $K(T)$ is open,

\item $\setG\setminus K(T)$ is open.
\end{enumerate}
\end{lemma}

\begin{proof}
If $Z=\emptyset$, define $K(\emptyset)=\emptyset$; we are done.
So assume from now on that $Z$, and hence $\setG$, are non-empty,
so that $\setG$ is a subspace of $X$.
Since $\setG\cap\did Z=\emptyset$,
it follows that  $\c O^+=\{Q\subseteq\setG:Q$ open, $|Q\cap Z|\leq1\}$ is an open cover of the subspace $\setG$.
This subspace, being a metric space, is paracompact 
--- see \cite{S48} or \cite[5.1.3]{Engelking89}.
So there is a locally finite open cover $\c O$ of $\setG$
that refines  $\c O^+$.
Evidently,
\begin{equation}\label{e:int leq1}
|O\cap Z|\leq1\mbox{ for every }O\in\c O.
\end{equation}

For each $z\in Z$, use 0-\dm ality to
choose a clopen \nhd\ $K^+(z)$ of $z$ contained in some $O\in\c O$ 
and disjoint from all but finitely many sets in $\c O$.
Since $z\in K^+(z)$, it follows from~\eqref{e:int leq1}
that each $O\in\c O$ contains at most one set $K^+(z)$.
Since $K^+(z)$ intersects only finitely many sets in $\c O$,
it therefore intersects only finitely many $K^+(t)$  $(t\in Z\setminus\{z\})$.
The union of these finitely many sets is clopen, so the set
 \[
K(z)=K^+(z)\setminus\bigcup_{t\in Z\setminus\{z\}}K^+(t)
\]
is clopen.  It also follows from~\eqref{e:int leq1} that $K^+(z)$ is the only $K^+(t)$
that contains $z$; so $z\in K(z)$. 

For each $T\subseteq Z$ define $K(T)=\bigcup_{t\in T}K(t)$.
We prove the lemma under this definition.
Items 1 and 2 are trivial. 
Item 3 holds because  the $K(z)$ $(z\in Z)$ are plainly \pd.
Item 4~holds because by definition, $K(T)$ is a union of open sets.
For item 5, see \cite[20.4--5]{Will70}, or prove it directly as follows.
Each $x\in \setG\setminus K(T)$
has an open \nhd\ $U$ such that $\{O\in\c O:U\cap O\neq\emptyset\}$ is finite,
and hence also
$\{t\in T:U\cap K(t)\neq\emptyset\}$ is finite.
Since a finite union of sets $K(t)$ is closed, and $x\notin K(T)$, the set $U\setminus K(T)$ is 
an open \nhd\ of $x$.
It follows that $\setG\setminus K(T)$ is open.
\end{proof}

The following is the first  result needed later, and is where
non-compactness comes in.

\begin{theorem}\label{thm:inf partition noncompact}
 $X$ is not compact iff $X$ can be partitioned into infinitely many non-empty open sets.
\end{theorem}

\begin{proof}
$\Leftarrow$ is obvious.
For $\To$, assume that $X$ is not compact.
By \cite[3.10.3]{Engelking89},
there is an infinite subset $Z\subseteq X$
with $\did Z=\emptyset$.
Taking $\setG$ in  lemma~\ref{lem:aiml mainly} to be $X$,
the lemma tells us that
$X$ is partitioned into the  \pd\ open sets
$K(\{z\})$ $(z\in Z)$ and $X\setminus K(Z)$.
The non-empty sets among these
(all but perhaps $X\setminus K(Z)$) form the required partition.
\end{proof}

\noindent By grouping sets together, an infinite partition into open sets can 
be `coarsened' into a partition into any finite number of open sets.

The next corollary is similar. Cf.~\cite[theorem 7.5]{gold:tang16}.
\begin{corollary}\label{cor:0dim partition}
Let $\setG\subseteq X$ be open, and $I$ be non-empty and countable.
Then $\setG$ can be partitioned into open sets $\setG_i$ $(i\in I)$
such that
$\cl(\setG)\setminus\setG=\cl(\setG_i)\setminus\setG_i$ for each $i\in I$.
\end{corollary}

\begin{proof}
By lemma~\ref{lem:simple Zn}, we can select \pd\ sets $\setI_i\subseteq\setG$ for $i\in I$, with $\did \setI_i=\cl\setG\setminus\setG$ for every  $i\in I$, and
$\setG\cap\did Z=\emptyset$, where $Z=\bigcup_{i\in I}\setI_i$.
Choose sets $K(T)\subseteq\setG$  (for $T\subseteq Z$)  
as in lemma~\ref{lem:aiml mainly}.
Fix any $i_0\in I$.
For each $i\in I$ let 
\[
\setG_i=
\begin{cases}
K(\setI_i),&\mbox{if }i\neq i_0,
\\
\setG\setminus\bigcup_{j\in I\setminus\{i_0\}}\setG_j
=\setG\setminus K(Z\setminus\setI_{i_0}),&\mbox{if }i=i_0.
\end{cases}
\]
By lemma~\ref{lem:aiml mainly}, the
$\setG_i$ are \pd\ open subsets of $\setG$,
and they plainly partition~$\setG$.

Let $i\in I$.
We check that $\cl\setG\setminus\setG=\cl\setG_i\setminus\setG_i$.
Notice that  $\setI_i\subseteq\setG_i$, even when $i=i_0$.
So 
$\cl\setG\setminus\setG=\did \setI_i\subseteq\cl\setG_i$.
Since $\cl\setG\setminus\setG$ is disjoint from $\setG$ and hence also from $\setG_i$,
we obtain $\cl\setG\setminus\setG\subseteq\cl\setG_i\setminus\setG_i$.

Conversely,
of course $\setG_i\subseteq\setG$,
so $\cl\setG_i\subseteq\cl\setG$.
Now  $\bigcup_{j\in I\setminus\{i\}}\setG_j$ is open and disjoint from $\setG_i$,
so it is also disjoint from $\cl\setG_i$.
Hence, $\cl\setG_i\setminus\setG_i$
is disjoint from $\setG_i\cup\bigcup_{j\in I\setminus\{i\}}\setG_j=\setG$.
We obtain $\cl\setG_i\setminus\setG_i\subseteq\cl\setG\setminus\setG$ as required.
\end{proof}

Now we come to the second result needed later.
The first part is equivalent to Tarski's well-known  `dissection theorem'
(\cite{Tar38}, later strengthened in \cite{McKiT44}),
except that $I$ can be infinite. The 
second part is distinctively 0-\dm al:
for example, the theorem 
can fail when $X=\R$ and $|I|\geq3$.
The third part harks back to the `$\varepsilon$~clause' 
in~\cite[lemma 4.3]{Kremer2010:stro}.

\begin{theorem}\label{thm:0-dim dissection}
For any 
non-empty countable set $I$
and any $\varepsilon>0$,
any non-empty open subset
$\setG\subseteq X$ can be partitioned into 
a non-empty set $\setB$ and 
(necessarily non-empty) open sets $\setG_i$
($i\in I$)  such that
\begin{enumerate}
\item $\displaystyle
\cl(\setG)\setminus\bigcup_{i\in I}\setG_i
=\cl\setB=\cl\setG_i\setminus\setG_i$ for each $i\in I$,

\item $\setG\cap\did\setB=\emptyset$,

\item  $d(x,\setB)<\varepsilon$ for every $x\in\setG$.

\end{enumerate}
\end{theorem}

\begin{proof}
Using Zorn's lemma, choose a maximal $\varepsilon$-sparse set $Z\subseteq\setG$.
Then 
$\did Z=\emptyset$,  $Z$ is non-empty (since any singleton 
subset of $\setG$ is $\varepsilon$-sparse), and
$d(x,Z)<\varepsilon$ for every $x\in\setG$ (else $x$ could be added to $Z$,
contradicting its maximality).

Now use lemma~\ref{lem:simple Zn} with $I$ a singleton
to choose $\setI\subseteq\setG$
such that
$\did\setI=\cl\setG\setminus\setG$,
and define $\setB=\setI\cup Z\subseteq\setG$.
Then 
\[
\did\setB=\did\setI\cup\did Z=\cl\setG\setminus\setG.
\]
Since $Z\subseteq\setB$, we have $\setB\neq\emptyset$ and
$d(x,\setB)<\varepsilon$ for every $x\in\setG$.
So parts 2 and 3 hold.

Let $\setG'=\setG\setminus \setB$.
Note that
$\did\setB$ is disjoint from $\setG$.
So by fact~\ref{fact:from apal},
$\setG'=\setG\setminus(\setB\cup\did\setB)=\setG\setminus\cl\setB$,
which  is open;  
$\int\setB\subseteq\setB\cap\did\setB=\emptyset$, so $\setB$ has empty interior;
hence
$\cl\setG'=\cl\setG$.
We now  use corollary~\ref{cor:0dim partition}
to partition $\setG'$
into open sets $\setG_i$ $(i\in I)$
with $\cl\setG_i\setminus\setG_i=\cl\setG'\setminus\setG'$ for each $i\in I$.
Then
\[
\cl\setG\setminus\bigcup_{i\in I}\setG_i
=\cl\setG\setminus\setG'=
\begin{cases}
(\cl\setG\setminus\setG)\cup(\setG\setminus\setG')
=\did \setB\cup\setB=\cl\setB,
\\
\cl\setG'\setminus\setG'=\cl\setG_i\setminus\setG_i\mbox{ for each }i\in I.
\end{cases}
\]
Each $\setG_i$ is non-empty since $\setB\subseteq\cl\setG_i$.
This proves part 1, and we are done.
\end{proof}

\penalty-400

\subsection{Logic}

\def\lad{\c L^{\forall}_{\bod}}

\begin{theorem}\label{thm:stro compl noncompact}
The language $\lad$ is compact over 
every non-compact 0-\dm al \dims.

\end{theorem}

\begin{proof}
We adopt a broadly similar approach to \cite{Kremer2010:stro}, and extend it to 
handle $\forall$
and $\bod$.
Fix a non-compact 0-\dm al \dims\
$X$ and a set $\Sigma$ of $\lad$-formulas
that is finitely satisfiable in $X$.
We show that $\Sigma$ is satisfiable in $X$.

\paragraph{Step 1.}
By the argument of \cite[theorem 8.4]{gold:tang16}
and the comments after it,
in the language $\lad$ 
the system KD4U is sound and complete over $X$.
Since $\Sigma$  is finitely satisfiable in $X$, it is KD4U-consistent.
Hence, using the canonical model and the downward L\"owenheim--Skolem theorem,
which are standard modal techniques,
we can find a countable Kripke model
$\c M=(W,R,h)$ whose frame $(W,R)$ validates KD4 and so is serial and transitive, and $w_0\in W$, such that $\c M,w_0\models\Sigma$.
(The U axioms are used in obtaining $\c M$.)

\paragraph{Step 2.}
We now define by induction on $n<\omega$ a set $\c G_n$ of \pd\ non-empty open subsets of $X$, and a `labeling' map $\lambda_n:\c G_n\to W$.

Since $X$ is not compact,
we can use theorem~\ref{thm:inf partition noncompact} to
partition it into \pd\ non-empty open sets $O_w$ $(w\in W)$.
We define $\c G_0=\{O_w:w\in W\}$ and
$\lambda_0(O_w)=w$ for each $w\in W$.
Since the $O_w$ are \pd, $\lambda_0$ is well defined.

Let $n<\omega$
and suppose inductively that $\c G_n,\lambda_n$ have been defined.
Let $\setG\in\c G_n$, and suppose that $\lambda_n(\setG)=u$, say.
Use theorem~\ref{thm:0-dim dissection} to partition $\setG$
into  non-empty open sets $\setG_w$ ($w\in R(u)$)
and a non-empty set $\setB(\setG)$
with 
\begin{itemize}
\item $\cl\setG\setminus\bigcup_{w\in R(u)}\setG_w=\cl\setB(\setG)=\cl\setG_w\setminus\setG_w$
for each $w\in R(u)$,

\item  $\setG\cap\did\setB(\setG)=\emptyset$,

\item  $d(x,\setB(\setG))<1/2^{n+1}$ for every $x\in\setG$.

\end{itemize}
We can apply the theorem here because the frame $(W,R)$ is serial and
so $R(u)\neq\emptyset$.
Let $\c G_{n+1}=\{\setG_w:\setG\in\c G_n,\;w\in R(\lambda_n(\setG))\}$.
Also define $\lambda_{n+1}:\c G_{n+1}\to W$ by $\lambda_{n+1}(\setG_w)=w$.
This is well defined, because the elements of $\c G_n$
are \pd, so each $\setG_w$ gets into $\c G_{n+1}$ in only one way.

That completes the definition of the $\c G_n,\lambda_n$.
Let $\c G=\bigcup_{n<\omega}\c G_n$
and $\lambda=\bigcup_{n<\omega}\lambda_n$.
Then $(\c G,\supset)$ is a forest (that is, a disjoint union of trees)
with roots the $O_w$ and
whose branches all have height $\omega$.
Also, since $R$ is transitive, 
it follows that $\lambda:(\c G,\supset)\to (W,R)$
is a surjective p-morphism.

\paragraph{Step 3.}
For each $x\in X$, let $\c E(x)=\{\setG\in\c G:x\in\setG\}$.
This is either a branch of the forest $(\c G,\supset)$, or a
 finite initial
segment of such a branch.
It is non-empty, since there is $w\in W$ with $x\in O_w$, and then $O_w\in\c E(x)$.

Select an \uf\ $D_x$ on $\c E(x)$ as follows.
If $\c E(x) $ is finite, 
its $\subseteq$-minimal element is $\bigcap\c E(x)$, and we let $D_x$  be the principal
\uf\ 
$\{S\subseteq\c E(x):\bigcap\c E(x)\in S\}$.
If $\c E(x) $ is infinite, we let $D_x$  be any non-principal \uf\ on $\c E(x)$.
Now let
\[
\Gamma_x=\big\{\varphi\in\lad:\{\setG\in\c E(x):
\c M,\lambda(\setG)\models\varphi\}\in D_x\big\}.
\]
Observe that \begin{itemize}
\item [$(\dag)$] every $\varphi\in\Gamma_x$ is true in $\c M$
at some world of the form $\lambda(\setG)$ for some $\setG\in\c E(x)$,

\item  [$(\ddag)$]
if  $\setG\in\c G$ and $x\in\setB(\setG)$, then
$\bigcap\c E(x)=\setG$
and $\Gamma_x=\{\varphi\in\lad:\c M,\lambda(\setG)\models\varphi\}$.

\end{itemize}

\paragraph{Step 4.}
Define an assignment $g$ into $X$
by $g(p)=\{x\in X:p\in \Gamma_x\}$,
for each atom $p\in\Var$.

\paragraph{Step 5.}
We now prove a `truth lemma':
that for every $\varphi\in\lad$,
we have $\varphi\in\Gamma_x$ iff $(X,g),x\models\varphi$ for each $x\in X$.

The proof is by induction on $\varphi$.
For  $\varphi\in\Var$ it holds by definition of $g$,
and the boolean cases
(including $\top$) follow from the fact that every $D_x$ is an \uf.

For the remaining cases, assume the result for $\varphi$ inductively, and let $x\in X$ be given.

For the case $\forall\varphi$,
if $\forall\varphi\in\Gamma_x$
then by $(\dag)$,
$\forall\varphi$ is true at some world of $\c M$,
so $\varphi$ is true at every world of $\c M$.
It follows from the definition of $\Gamma_y$ that $\varphi\in\Gamma_y$,
 and inductively that
$(X,g),y\models\varphi$, for every $y\in X$.
So $(X,g),x\models\forall\varphi$.

Conversely, suppose that $(X,g),x\models\forall\varphi$.
Let $w\in W$ be given.
Choose any $y\in\setB(O_w)$.
Then $(X,g),y\models\varphi$,
so inductively, $\varphi\in\Gamma_y$.
By $(\ddag)$ and because $\lambda(O_w)=w$, we get $\c M,w\models\varphi$.
As $w$ was arbitrary, we get $\c M,w\models\forall\varphi$ for every $w\in W$.
It is now immediate from the definition of $\Gamma_x$ that
$\forall\varphi\in\Gamma_x$.

Finally we consider the case $\bod\varphi$.
Suppose first that $\bod\varphi\in\Gamma_x$.
By $(\dag)$,
there is $\setG\in\c E(x)$ 
with $\c M,\lambda(\setG)\models\bod\varphi$.
Then for every $y\in\setG\setminus\setB(\setG)$,
the set $S=\{\setG'\in\c E(y):\setG'\subsetneq\setG\}$
is in $D_y$ by choice of $D_y$.
Also, every $\setG'\in S$ 
satisfies
$R(\lambda(\setG),\lambda(\setG'))$ as $\lambda$
is a p-morphism (again we need transitivity of $R$ here), 
and so $\c M,\lambda(\setG')\models\varphi$ by Kripke semantics.
So $\varphi\in\Gamma_y$ by definition of $\Gamma_y$,
and inductively, $(X,g),y\models\varphi$,
for every $y\in\setG\setminus\setB(\setG)$.

Now $x\in\setG$.
If $x\notin\setB(\setG)$, then $\setG\setminus\setB(\setG)$ is already
an open \nhd\ of $x$ all of whose elements satisfy $\varphi$.
If $x\in\setB(\setG)$, then
recalling that $\setG\cap\did\setB(\setG)=\emptyset$,
we can find an open \nhd\ $O$ of $x$ with $O\subseteq\setG$
and $O\cap\setB(\setG)=\{x\}$.
By the above, $(X,g),y\models\varphi$ for every $y\in O\setminus\{x\}$.
Either way, we have shown that $(X,g),x\models\bod\varphi$.

Conversely, suppose that $(X,g),x\models\bod\varphi$.
So there is $\varepsilon>0$ such
that $(X,g),y\models\varphi$ 
for every $y\in N_\varepsilon(x)\setminus\{x\}$.
We show that $\bod\varphi\in\Gamma_x$.

Suppose first that $\c E(x)$ is finite, with least element 
$\bigcap\c E(x)=\setG$, say.
Then $x\in\setB(\setG)$,
so by $(\ddag)$ it suffices to show
$\c M,\lambda(\setG)\models\bod\varphi$.
Accordingly, take any $w\in R(\lambda(\setG))$.
We show that $\c M,w\models\varphi$.
Now 
\[
x\in\setB(\setG)
\subseteq\cl\setB(\setG)
=\cl\setG_w\setminus\setG_w
\subseteq\cl\setG_w\setminus\bigcup_{u\in R(w)}(\setG_w)_u
=\cl\setB(\setG_w).
\]
And $x\notin\setB(\setG_w)$ since $\setB(\setG)$ is disjoint from $\setG_w$.
So there is $y\in \setB(\setG_w)\cap N_\varepsilon(x)\setminus\{x\}$.
For such a $y$ we have $(X,g),y\models\varphi$,
so inductively, $\varphi\in\Gamma_y$,
and by $(\ddag)$ we obtain $\c M,w\models\varphi$ since $\lambda(\setG_w)=w$.
We are done.

Now suppose instead that $\c E(x)$ is infinite.
Let
\[
S=\c E(x)\cap\bigcup\{\c G_n:0<n<\omega,\;1/2^n<\varepsilon\},
\]
a cofinite subset of $\c E(x)$.
Pick arbitrary $\setG\in S$.
We show that $\c M,\lambda(\setG)\models\bod\varphi$.
Suppose $\setG\in\c G_n$.
By choice of $\setB(\setG)$
we have $d(x,\setB(\setG))<1/2^n<\varepsilon$.
Now $x\notin\setB(\setG)$ since $\c E(x)$ is infinite.
So there is $y\in \setB(\setG)\cap N_\varepsilon(x)\setminus\{x\}$.
Then $N_\varepsilon(x)\setminus\{x\}$
is an open \nhd\ of $y$, and every $z\in N_\varepsilon(x)\setminus\{x\}$
satisfies $(X,g),z\models\varphi$.
So $(X,g),y\models\bod\varphi$.
Since $y\in\setB(\setG)$, $\c E(y)$ is finite,
so by the proof above we have 
$\c M,\lambda(\setG)\models\bod\varphi$ as required.

We have shown that
each $\setG\in S$
satisfies $\c M,\lambda(\setG)\models\bod\varphi$.
Since $S$ is cofinite in $\c E(x)$, it is certainly in $D_x$,
and it follows by definition of $\Gamma_x$ that $\bod\varphi\in\Gamma_x$ as required.

\paragraph{Step 6.}
Recall that $\c M,w_0\models\Sigma$.
Take any  $x\in \setB(O_{w_0})$.
By $(\ddag)$,
$\Sigma\subseteq\Gamma_x$, so by step 5 (the truth lemma)  above, 
$(X,g),x\models\Sigma$.
So $\Sigma$ is satisfiable in $X$.
\end{proof}

\begin{corollary}\label{cor:str compl noncomp}
Let $X$ be a non-compact 0-\dm al \dims.
In the language $\lad$,
the system KD4U is sound and strongly complete
over $X$.
In the weaker language $\c L^\forall_\bo$,
the system S4U is sound and strongly complete
over $X$.
\end{corollary}

\begin{proof}
S4U and KD4U are outlined in \S\ref{ss:intro of here}
and defined fully in, e.g., \cite{gold:tang16}
and \cite[\S8.1]{goldhod:tang17:apal}.
As shown in the former (in particular by theorem 5.1, the argument of theorem 8.4,
and the discussion following it),
they are sound and complete over every 0-\dm al \dims\ in their respective languages.
The corollary now follows by theorem~\ref{thm:stro compl noncompact}
and fact~\ref{fact:stro compl = compact}.
\end{proof}

\section{Cantor set}\label{sec:cantor set}
In the preceding section
we  proved strong completeness of the system KD4U in the language $\lad$ over
every non-compact 0-\dm al \dims.
Actually, this covers all 0-\dm al \dimss\
 except one --- the \emph{Cantor set.}
 
The Cantor set is,
up to homeomorphism, the
unique compact 0-\dm al \dims\
(see \cite{Brouw1910} or \cite[29.5, 30.4]{Will70}).
As a topological space, 
it is the Stone space of the countable atomless \ba\
(see \cite[theorem 6.6 and text after corollary 7.7]{BellSlom69} or
 \cite[example 7.24]{Kopp89}).

In this section we show that, over the Cantor set, compactness
fails  for $\c L_{\bod}^{\forall}$
--- in surprising contrast to non-compact spaces ---
but holds for $\lbc$.
Compactness for the weaker
languages $\c L^\forall_\bo$ and $\c L_\bo^{\bdf}$ follows immediately,
and here we also obtain  strong completeness results.
We have none for $\lbc$ itself only because we do not know the
logic of the Cantor set in this language.

\subsection{Strong completeness fails  with $\bod,\forall$}

We start by observing that the results for non-compact
 spaces of the preceding section cannot be replicated for the Cantor set.

\begin{theorem}\label{thm:str compl fails cantor [d],A}
Let $X$ be an infinite compact 
T1 topological space (such as the Cantor set).
The 
language $\c L_{\bod}^\forall$ is not compact over $X$.
Hence,
in $\c L_{\bod}^\forall$ or any stronger language, no  deductive system is
sound and strongly complete over $X$.
\end{theorem}

\begin{proof}
We write down $\c L_{\bod}^\forall$-formulas saying that 
the valuation of an atom is infinite but has empty derivative.
Let $p_0,p_1,\ldots,q\in\Var$ be pairwise distinct, and let
\[
\Sigma=\{\exists(q\wedge p_i\wedge\bigwedge_{j<i}\neg p_j):i<\omega\}
\cup\{\forall\neg\did q\}.
\]
Any finite subset of $\Sigma$ is satisfiable in $X$:
if the subset involves only $\Vec pn,q$, 
choose pairwise distinct points
$\Vec xn\in X$, assign each $p_i$ to $\{x_i\}$, and  $q$ to $\{\Vec xn\}$.
No point satisfies $\did q$, since 
in a T1 space, every finite set has empty derivative (and conversely).

Suppose for contradiction that $\Sigma$ as a whole were satisfiable in $(X,h)$ for some assignment
$h$ into $X$.
For each $i<\omega$ pick $x_i\in X$
with $(X,h),x_i\models q\wedge p_i\wedge\bigwedge_{j<i}\neg p_j$.
The $x_i$ are plainly pairwise distinct,
and hence $h(q)$ is infinite.
Since $X$ is compact, by \cite[3.10.3]{Engelking89}
every infinite subset of $X$ has non-empty derivative.
So there is $x\in\did h(q)$, and therefore $(X,h),x\models\did q$,
contradicting the truth of $\forall\neg\did q$ in $(X,h)$.

So $\c L_{\bod}^\forall$ is not compact over $X$,
proving the first part of the theorem.
The second part follows by fact~\ref{fact:stro compl = compact}.
\end{proof}

The proof really needs $\did$: using $\forall\neg\di q$
in $\Sigma$ instead loses finite satisfiability,
since even $\{\exists q,\forall\neg\di q\}$ is not satisfiable.
In theorem~\ref{thm:stro compl cantor bo A} we will show that
the result needs $\did$ too.

\subsection{Compactness holds with $\bo,\ang n$ for $n<\omega$}

Replacing $\bod$ by the weaker connective $\bo$, we have more success.
In fact, we will prove compactness for $\lbc$ over the Cantor set. 
Our proof uses a third kind of compactness --- in \fo\ logic.
Every consistent set of \fo\ sentences has a model.

\subsubsection{Two-sorted \fo\ \str s}\label{ss:2sorted}

To formulate topological models in \fo\ logic,
we introduce a two-sorted \fo\ signature~$L$.
It has a `point' sort and a `set' sort,
so $L$-\str s have the form $M=(X,B)$,
where $X$ is the set of elements of $M$ of point sort,
and $B$ is the set of elements of set sort.
The symbols of $L$ comprise a binary relation symbol $\in$ 
relating points to sets,
`boolean' function symbols $+$ (binary) and $-$ (unary), and constants $0,1$, all acting on the set sort,
and a unary relation symbol $P$ of point sort for each $p\in\Var$.
For convenience, we also include in $L$ a  point-sorted constant $k$.
As usual, we write $s^M$ for the interpretation of a symbol $s$ of $L$ in 
an $L$-\str\ $M$.
We use $x,y,z,\ldots$ for point-sorted variables
(and also by abuse for point-sorted elements),
and $b,c,o,O,\ldots$ for set-sorted variables
(and also by abuse for elements of set sort).

Given an $L$-\str\ $M=(X,B)$,
for each $b\in B$ we let
$\check b=\{x\in X:M\models x\in b\}\subseteq X$.
It may be that $\check b=\check c$ for distinct $b,c\in B$, but this will not happen in our applications.

We can view a topological model as an $L$-\str\ as follows.
Let $X$ be a  0-\dm al topological space
and write $\clop(X)$ for the set of all clopen subsets of $X$.
This is a base for the topology on $X$, and 
$(\clop(X),\cup,\unarycomplement,\emptyset,X)$ is a boolean set algebra.
Let $h:\Var\to \wp(X)$ be an assignment.
Then the topological model $(X,h)$
can be turned into a two-sorted $L$-\str\ $(X,h)^{(2)}=M$, say, 
where $M$ has the form $(X,\clop(X))$,
$\in$
is interpreted in $M$ as ordinary set membership,
the boolean operations are interpreted 
as $b+c=b\cup c$, $-b=X\setminus b$, $0=\emptyset$, and $1=X$,
the constant $k$ has arbitrary interpretation in $X$, and $P^M=h(p)$ for each $p\in\Var$.
The \str\ $(X,h)^{(2)}$ is not unique: it depends on the
interpretation of $k$.
Each $b\in \clop(X)$ is both a set-sorted element of $M$
and a set of point-sorted elements of $M$, and
by definition of $\in^M$ we have $b=\check b\subseteq X$.
So we often do not need to write $\check b$ 
when dealing with `concrete' \str s like this
(the proof of lemma~\ref{lem: C good} is an example).

Conversely, given an $L$-\str\ $M=(X,B)$,
we endow $X$ with the topology generated by $\check B=\{\check b:b\in B\}$.
Define an assignment $h:\Var\to\wp(X)$
by $h(p)=P^M\subseteq X$ for each $p\in\Var$.
We end up with a topological model $M^{(1)}=(X,h)$, where $X$ is the topological space just defined.
Plainly, if $X$ is 0-\dm al  then $((X,h)^{(2)})^{(1)}=(X,h)$ for any $h$.

\subsubsection{Standard translation}\label{ss:std tr}
Every $\lbc$-formula $\varphi$
has a `standard translation' to an $L$-formula $\varphi^x$,
for any \fo\ variable $x$ of point sort.
The translation $\varphi^x$ will have at most the variable $x$ free.
We define $\varphi^x$ by induction on $\varphi$:
\begin{itemize}
\item $p^x=P(x)$ for $p\in\Var$
\item $\top^x=\top$
\item $(\neg\varphi)^x=\neg\varphi^x$,
and $(\varphi\wedge\psi)^x=\varphi^x\wedge\psi^x$

\item $(\bo\varphi)^x=\exists O(x\in O\wedge\forall y(y\in O\to\varphi^y))$

\item $(\ang n\varphi)^x=\Exists_{0\leq i\leq n}x_i(\bigwedge_{i<j\leq n}x_i\neq x_j
\wedge\bigwedge_{i\leq n}\varphi^{x_i})$, for $n<\omega$.

\end{itemize}

As one might expect,
$\varphi^x$ generally
`means the same' as $\varphi$, as the following lemma shows.
In the lemma and later,
$M\models\varphi^x(a)$
 means that $\varphi^x$ is true in $M$ when $x$ is assigned to $a$,
and $\varphi^x(k/x)$ denotes the
$L$-sentence obtained by substituting the constant $k$
for every free occurrence of $x$ in $\varphi^x$.

\begin{lemma}\label{lem:std tr}
Let a topological model $(X,h)$ and
an $L$-\str\ $M=(X,B)$ be given,
and suppose that 
$\check B=\{\check b:b\in B\}$ is a base for the topology on $X$.
Then for every $\varphi\in\lbc$ and $a\in X$
we have $(X,h),a\models\varphi$ iff $M\models\varphi^x(a)$,
and hence 
$(X,h),k^M\models\varphi$ iff $M\models\varphi^x(k/x)$.
\end{lemma}

\begin{proof}[Proof (sketch)]
The proof is by induction on $\varphi$.
We consider only the case $\bo\varphi$, as the other cases are straightforward.
Let $a\in X$.
Then $(X,h),a\models\bo\varphi$ iff
$a$ has an open \nhd\ $O$
with $(X,h),a'\models\varphi$ for every $a'\in O$.
As $\check B$ is a base for the topology on $X$,
and by the inductive hypothesis,
this is iff there is $b\in B$
with $ a\in \check b$
and $M\models \varphi^x(a')$ for every $a'\in \check b$.
This is plainly iff $M\models(\bo\varphi)^x(a)$.
\end{proof}

We will prove that $\lbc$
is compact over the Cantor set using standard translations,
which give us access to \fo\ compactness.
Suppose that $\Sigma$ is a set of $\lbc$-formulas
that is finitely satisfiable over the Cantor set.
It will follow that for a certain \fo\ theory $T$,
the theory $T\cup\{\varphi^x(k/x):\varphi\in\Sigma\}$
is consistent, so by 
\fo\ compactness, it has a model.
We will transform a countable model of it into a model of $\Sigma$
over the Cantor set.
The `side theory' $T$ allows us to do this.
It will in fact be the theory $T_{good}$, defined next.

\subsubsection{Good $L$-\str s}

For set-sorted terms $b,c$, we write
$b\leq c$ to abbreviate the $L$-formula $b+c=c$,
and for any $\lbc$-formula $\varphi$,
we write $b\subseteq\sem\varphi$ 
to abbreviate the $L$-formula
$\forall x(x\in b\to\varphi^x)$.

\begin{definition}\label{def:good}
An $L$-\str\ $M=(X,B)$  is said to be \emph{good}
if \begin{enumerate}
\item\label{good:atomlessba} $(B,+^M,-^M,0^M,1^M)$ is an atomless \ba\

\item\label{good:rep}
$M\models\forall bcx([x\in b+c\leftrightarrow x\in b\vee x\in c]
\wedge [x\in -b\leftrightarrow \neg(x\in b)])$

\item\label{good:hausdorff} $M\models\forall bc(\forall x(x\in b\leftrightarrow x\in c)\to b=c)$

\item\label{good4} $M\models\forall xy(\forall b(x\in b\leftrightarrow y\in b)\to x=y)$

\item\label{good:main} $\displaystyle M\models\forall b\Big(b\subseteq\sem{\bigvee_{\psi\in\Psi}\bo\psi}
\to\Exists_{\psi\in\Psi}c_\psi
\Big((b\leq\sum_{\psi\in\Psi}c_\psi)
\wedge\bigwedge_{\psi\in\Psi}(c_\psi\subseteq\sem{\bo\psi})\Big)\Big)$,

for every non-empty finite set $\Psi$ of $\lbc$-formulas.

\end{enumerate}
Let $T_{good}$ be the \fo\ $L$-theory comprising 
\fo\ sentences expressing clause~\ref{good:atomlessba} and the $L$-sentences from 
clauses \ref{good:rep}--\ref{good:main} above.
\end{definition}

An $L$-\str\ $M$ is good iff $M\models T_{good}$.
Let us give some examples of good and `bad' $L$-\str s.
Good \str s arise from topological models over 
the Cantor set, and more generally over any
separable 0-\dm al \dims:

\begin{lemma}\label{lem: C good}
Let $X$ be a separable 0-\dm al \dims,
let $(X,h)$ be any topological model over $X$,
and let $M=(X,\clop(X))=(X,h)^{(2)}$ be an $L$-\str\ 
derived from $(X,h)$ as described in \S\ref{ss:2sorted}.
Then $M$ is good.
\end{lemma}

\begin{proof}
As $X$ is 0-\dm al and 
dense in itself, $(\clop(X),\cup,\unarycomplement,\emptyset,X)$ is an atomless \ba,
and clauses \ref{good:rep}--\ref{good4} of definition~\ref{def:good} clearly hold for $M$.

We check clause~\ref{good:main}.
For a $\lbc$-formula $\varphi$, write $\sem\varphi^X$
for $\{x\in X:(X,h),x\models\varphi\}$.
Let $b\in \clop(X)$ and let a non-empty finite set $\Psi$ of $\lbc$-formulas
be given, with
$b\subseteq\sem{\bigvee_{\psi\in\Psi}\bo\psi}^X$.

Now we use some topology.
As $X$ is a separable \ms, it is Lindel\"of
\cite[4.1.16]{Engelking89}.
As it is also 0-\dm al, by \cite[6.2.5, 6.2.7]{Engelking89}
the finite open cover $\{-b\}\cup\{\sem{\bo\psi}^X:\psi\in\Psi\}$
of $X$ can be refined to 
a cover consisting of \pd\ open sets.
Plainly, any union of these sets is clopen.
So we can find clopen sets $c_\psi\in\clop(X)$ with $c_\psi\subseteq\sem{\bo\psi}^X$
(for each $\psi\in\Psi$)
such that $b\subseteq\bigcup_{\psi\in\Psi}c_\psi$.

Clause~\ref{good:main} now follows by
clauses 1--3 and lemma~\ref{lem:std tr},
which applies since $\clop(X)$ is a base for the topology on $X$.
So $M$ is good.
\end{proof}

\begin{example}\label{eg:bad str}\rm
An example of a bad $L$-\str\ 
is $Q=(\rats,B)$, where
$\rats$ is the set of rational numbers,
$B$ is the countable atomless \ba\ consisting of finite unions of intervals
of $\rats$
of the form  $(x+\pi,y+\pi)$
(where $x<y$ in $\rats\cup\{\pm\infty\}$),
$\in^Q$ is ordinary set membership,
and for some atom $p\in L$ we have
\[
P^Q=\bigcup_{n\in\mathbb{Z}}(2n+\pi,2n+1+\pi),
\]
where $\ints$ denotes the set of integers.
Under the standard metric $d(x,y)=|x-y|$, $\rats$ is a separable
0-\dm al \dims,
and $B$ is a base of clopen sets for its topology.
However, $\rats$ has continuum-many clopen sets, and indeed
$P^Q$ is clopen but is not in $B$.  So $B\subsetneq\clop(\rats)$.

Now $\rats\in B$ and
$\rats\subseteq \sem{\bo p\vee\bo\neg p}^X$.
But the sets  $\sem{\bo p}^X$ and $\sem{\bo\neg p}^X$
(which are $P^Q$ and $\rats\setminus P^Q$, respectively)  are disjoint.
So for any $c,c'\in B$, if $\rats\subseteq c\cup c'$, $c\subseteq\sem{\bo p}^X$, and $c'\subseteq\sem{\bo\neg p}^X$, then in fact $c=P^Q$, which is impossible since $P^Q\notin B$.
So there are no such $c,c'$, and 
clause~\ref{good:main} of definition~\ref{def:good} fails.
(The other clauses are ok.)
\end{example}

\subsubsection{Ultrafilter extensions of good \str s}
We now aim to construct an `\uf\ extension' of a good \str.
In \S\ref{sss:models cantor}, we will show that for a {countable} good \str, 
this extension is homeomorphic to the Cantor set, and `truth-preserving'.

So until the end of \S\ref{sss:models cantor}, fix a good $L$-\str\ $M=(X,B)$.
Then $(B,+^M,-^M,0^M,1^M)$ is an atomless \ba, which we write henceforth simply as $B$.
We write
\[
\begin{array}{rcll}
\check b&=&\{x\in X:M\models x\in b\}&\mbox{for }b\in B,
\\
\hat x&=&\{b\in B:M\models x\in b\}&\mbox{for }x\in X.
\end{array}
\]
Each $\hat x$ is a (non-principal) \uf\ of $B$.
By clauses \ref{good:rep}--\ref{good:hausdorff} of definition~\ref{def:good}, 
the map $(b\mapsto\check b)$
is a boolean embedding of  $B$ into 
the boolean set algebra $(\wp(X),\cup,\unarycomplement,\emptyset,X)$.
We form the topological model $M^{(1)}=(X,h)$ as outlined
in \S\ref{ss:2sorted} above.
Then $\check B=\{\check b:b\in B\}$
contains $X$ and
is closed under finite intersections, and hence
\cite[5.3]{Will70} is a base for the topology on $X$.
So lemma~\ref{lem:std tr} applies to $M$ and $(X,h)$.
We have $\check B\subseteq\clop(X)$, but the inclusion may be proper
(see example~\ref{eg:bad str}).

We will let $\varphi,\psi,$ etc., denote arbitrary $\lbc$-formulas.
We write $\sem\varphi^X=\{x\in X:(X,h),x\allowbreak\models\varphi\}$.

\begin{definition}
Let $\mu$ be an \uf\ of $B$.
For a $\lbc$-formula $\varphi$,
we write
\begin{enumerate}

\item[$\bullet$] $\mu\Vdash\bo\varphi$ if there is $b\in\mu$
such that $\check b\subseteq\sem\varphi^X$,

\item[$\bullet$] $\mu\Vdash\di\varphi$ if $\mu\not\Vdash\bo\neg\varphi$.

\end{enumerate}
We define $F_\mu=\{\check b:b\in\mu\}\cup\{\sem{\di\psi}^X:\mu\Vdash\di\psi\}
\cup\{\{x\}:x\in X,\;\mu=\hat x\}$.
So $F_\mu\subseteq\wp(X)$.
\end{definition}

\begin{lemma}\label{lem:F fs}
For each \uf\ $\mu$ of $B$, the set
$F_\mu$ has the finite intersection property
(i.e., $\bigcap S\neq\emptyset$ for every
non-empty finite $S\subseteq F_\mu$).
\end{lemma}

\begin{proof}
Suppose first that $\mu=\hat x$ for some $x\in X$.
Then $x$ is unique (by clause~\ref{good4} of definition~\ref{def:good}),  
$x\in\check b$ for every $b\in\mu$,
and $x\in \sem{\di\psi}^X$ for every $\psi$ with $\mu\Vdash\di\psi$.
So $x\in\bigcap F_\mu$ and we are done.

Now suppose that there is no such $x$,
so $F_\mu=\{\check b:b\in\mu\}\cup\{\sem{\di\psi}^X:\mu\Vdash\di\psi\}$.
As we said, lemma~\ref{lem:std tr} applies to $M$ and $(X,h)$,
so
\begin{equation}\label{e:std tr in M}
\check b\subseteq\sem\varphi^X \iff M\models b\subseteq\sem\varphi,
\quad\mbox{for each }b\in B.
\end{equation}

Assume for contradiction that
there are $b\in \mu$ and $\lbc$-formulas $\Vecc\psi n$
with   $\mu\Vdash\di\psi_i$ for each $i<n$,
such that $\check b\cap\bigcap_{i<n}\sem{\di\psi_i}^X=\emptyset$.
Hence, $n>0$ and
$\check b\subseteq\bigcup_{i<n}\sem{\bo\neg\psi_i}^X$.
Then~\eqref{e:std tr in M} gives
 $M\models b\subseteq\sem{\bigvee_{i<n}\bo\neg\psi_i}$.
Since $M$ is good, 
there are $\Vecc cn\in B$
with $M\models c_i\subseteq\sem{\bo\neg\psi_i}$ for each $i<n$,
and $M\models b\leq\sum_{i<n}c_i$.
But  $\mu$ is an \uf\ containing $b$,
so $c_i\in\mu$ for some $i<n$.
Using~\eqref{e:std tr in M} again,
$\check c_i\subseteq\sem{\bo\neg\psi_i}^X\subseteq\sem{\neg\psi_i}^X$, 
so $\mu\Vdash\bo\neg\psi_i$, contradicting
$\mu\Vdash\di\psi_i$.
\end{proof}

\begin{definition}\label{def:gammamu}
For each \uf\ $\mu$ of $B$, we choose
an \uf\ $\b\mu$ on $X$ containing $F_\mu$.
(By lemma~\ref{lem:F fs} and the boolean prime ideal theorem, this is possible.)
We then define
\[
\Gamma_\mu=\{\varphi\in\lbc:\sem\varphi^X\in\b\mu\}.
\]
\end{definition}

\begin{lemma}\label{lem:arrays}
Let $\mu$ be an \uf\ of $B$.
Then for all $\lbc$-formulas
$\varphi,\psi$, we have:
\begin{enumerate}
\item\label{arrays1} $\neg\varphi\in\Gamma_\mu$ iff $\varphi\notin\Gamma_\mu$.

\item\label{arrays2}  $\varphi\wedge\psi\in\Gamma_\mu$ iff $\{\varphi,\psi\}\subseteq\Gamma_\mu$.

\item\label{arrays3} $\bo\varphi\in\Gamma_\mu$ iff $\mu\Vdash\bo\varphi$.
Either condition implies $\varphi\in\Gamma_\mu$.

\item\label{arrays4} If $\varphi\in\Gamma_\mu$ then $\mu\Vdash\di\varphi$.

\item\label{arrays5} If $x\in X$ and $\mu=\hat x$,
then $\varphi\in\Gamma_\mu$ iff $(X,h),x\models\varphi$.
\end{enumerate}
\end{lemma}

\begin{proof}
\begin{enumerate}
\item [1, 2.] These hold since $\b\mu$ is an \uf\ on $X$.

\setcounter{enumi}{2}

\item If  $\mu\Vdash\bo\varphi$
then there is $b\in\mu$ with $\check b\subseteq\sem\varphi^X$.
But $\check b$ is open, so $\check b\subseteq\int\sem\varphi^X=\sem{\bo\varphi}^X$.
As $\check b\in F_\mu\subseteq\b\mu$,
we have  $\sem{\bo\varphi}^X\in \b\mu$ as well, and so $\bo\varphi\in\Gamma_\mu$.

Conversely, if $\bo\varphi\in\Gamma_\mu$,
then $\sem{\bo\varphi}^X\in\b\mu$.
Since $\b\mu$ is an \uf,
$\sem{\di\neg\varphi}^X\notin \b\mu\supseteq F_\mu$.
This means that $\mu\not\Vdash\di\neg\varphi$,
and hence clearly $\mu\Vdash\bo\varphi$.

In either case, $\sem{\bo\varphi}^X\in\b\mu$.
But $\sem{\bo\varphi}^X\subseteq\sem{\varphi}^X$.
So also $\sem{\varphi}^X\in\b\mu$, and
$\varphi\in\Gamma_\mu$.

\item This follows from \ref{arrays3} and \ref{arrays1}.

\item We have $\{x\}\in F_\mu\subseteq\b\mu$.
In this case, $\b\mu$ is principal.
So  $\varphi\in\Gamma_\mu$ iff $\sem\varphi^X\in\b\mu$,
iff $x\in\sem\varphi^X$, iff $(X,h),x\models\varphi$.
\end{enumerate}
\unskip
\end{proof}

\subsubsection{Models over Cantor set from 
countable good \str s}\label{sss:models cantor}

We now further assume that \emph{the \ba\ $B$ is countable.}
As $B$ is also atomless, 
its Stone space (of \uf s) is homeomorphic to the Cantor set $\setC$
(as pointed out at the start of \ref{sec:cantor set}),
and we will identify the two.
So we take $\setC$ to be the set of \uf s of $B$,
and the clopen sets in $\setC$ to be the sets of the form
$\{\mu\in\setC:b\in\mu\}$ for $b\in B$.
These sets form a base for the topology on $\setC$.

\begin{definition}\label{def:cantor asst}
Define an assignment $g:\Var\to\wp(\setC)$
by $g(p)=\{\mu\in\setC:p\in\Gamma_\mu\}$, for each atom $p\in\Var$.
Here, $\Gamma_\mu$ is as in definition~\ref{def:gammamu}.
\end{definition}

\begin{lemma}[truth lemma]\label{lemma:tl2}
For every $\lbc$-formula $\varphi$,
we have
\[
(\setC,g),\mu\models\varphi\iff\varphi\in\Gamma_\mu,
\quad\mbox{for every }\mu\in\setC.
\]
\end{lemma}

\begin{proof}
The proof is by induction on $\varphi$.
For  $\varphi\in\Var$ it follows from the definition of $g$,
and  obviously $\top\in\Gamma_\mu$.
The boolean cases $\neg\varphi$ and $\varphi\wedge\psi$ are easy, using lemma~\ref{lem:arrays}(\ref{arrays1},\ref{arrays2}).
For the remaining cases, assume the result for $\varphi$ inductively,
and first consider $\bo\varphi$.

Let $\mu\in \setC$ be given.
If $\bo\varphi\in\Gamma_\mu$, then $\mu\Vdash\bo\varphi$
by lemma~\ref{lem:arrays}(\ref{arrays3}), so there is $b\in\mu$ such that 
$\check b\subseteq\sem\varphi^X$.
So $\nu\Vdash\bo\varphi$ for every $\nu\in \setC$ with $b\in\nu$
---  $b$ itself witnesses this.
So by lemma~\ref{lem:arrays}(\ref{arrays3}) again,
$\varphi\in \Gamma_\nu$ for all such $\nu$.
Inductively, $(\setC,g),\nu\models\varphi$
for all such $\nu$.
The set of these $\nu$ is a clopen subset of $\setC$ containing $\mu$,
so  by semantics,
$(\setC,g),\mu\models\bo\varphi$.

Conversely,
if $(\setC,g),\mu\models\bo\varphi$
then there is $b\in\mu$
with  $(\setC,g),\nu\models\varphi$ for every $\nu\in\setC$ containing $b$.
Inductively, $\varphi\in\Gamma_\nu$ for all such $\nu$.
In particular, for every $x\in\check b$,
since $b\in\hat x$, we have $\varphi\in\Gamma_{\hat x}$.
By lemma~\ref{lem:arrays}(\ref{arrays5}), $(X,h),x\models\varphi$ for all such $x$.
So $\check b\subseteq\sem\varphi^X$, and thus $\mu\Vdash\bo\varphi$.
By lemma~\ref{lem:arrays}(\ref{arrays3}), $\bo\varphi\in\Gamma_\mu$.

Finally, let $n<\omega$
and consider the case $\ang n\varphi$.
Let $\mu\in \setC$ be given.
If $\ang n\varphi\in\Gamma_\mu$,
then $\sem{\ang n\varphi}^X\in\b\mu$,
so certainly
$\ang n\varphi$ is true at some point of $(X,h)$.
So there are more than $n$ points $x\in X$
at which $(X,h),x\models\varphi$.
For each such $x$
we have $\varphi\in\Gamma_{\hat x}$
by lemma~\ref{lem:arrays}(\ref{arrays5}),
so $(\setC,g),\hat x\models\varphi$ by the inductive hypothesis.
By clause~\ref{good4} of definition~\ref{def:good}, the $\hat x$ are pairwise distinct, so
 $(\setC,g),\mu\models\ang n\varphi$ by semantics.

Conversely suppose $(\setC,g),\mu\models\ang n\varphi$,
so there are pairwise distinct $\Vec\mu n\in\setC$
with $(\setC,g),\mu_i\models\varphi$,
and hence inductively $\varphi\in\Gamma_{\mu_i}$, for each $i\leq n$.
Using standard properties of \uf s, 
we can find elements $b_i\in\mu_i$ ($i\leq n$)
such that $\Vec{\check b}n$ are \pd.
For each $i\leq n$,
since $\varphi\in\Gamma_{\mu_i}$,
by lemma~\ref{lem:arrays}(\ref{arrays4}) we have  $\mu_i\Vdash\di\varphi$.
So since $b_i\in\mu_i$, there is $x_i\in\check b_i$
with $(X,h),x_i\models\varphi$.
The $x_i$ are plainly pairwise distinct,
so $\ang n\varphi$ is true in $(X,h)$ at every point.
Then $\sem{\ang n\varphi}^X=X\in\b\mu$, so 
$\ang n\varphi\in \Gamma_\mu$.
\end{proof}

It follows that $(\setC,g)$ `extends' $(X,h)$ in a truth-preserving way:

\begin{corollary}\label{cor:extn}
$(X,h),x\models\varphi$ iff $(\setC,g),\hat x\models\varphi$,
for every $\lbc$-formula $\varphi$ and $x\in X$.
\end{corollary}

\begin{proof}
By lemmas~\ref{lem:arrays}(\ref{arrays5}) and~\ref{lemma:tl2},
 $(X,h),x\models\varphi$ iff $\varphi\in\Gamma_{\hat x}$,  iff $(\setC,g),\hat x\models\varphi$.
\end{proof}

\subsubsection{Compactness and strong completeness}
We can now prove that $\lbc$ is compact over the Cantor set.

\begin{theorem}\label{thm:stro compl cantor bo A}
Every set $\Sigma$ of $\lbc$-formulas 
that is finitely satisfiable in the Cantor set $\setC$
 is satisfiable in $\setC$.
\end{theorem}

\begin{proof}
Define 
$U=T_{good}\cup\{\varphi^x(k/x):\varphi\in\Sigma\}$.

\claim
$U$ is consistent.

\pfclaim
Let $\Sigma_0\subseteq\Sigma$ be finite.
As $\Sigma$ is finitely satisfiable in $\setC$,
there are an assignment $h$ into $\setC$,
and a point $x\in\setC$,
with $(\setC,h),x\models\Sigma_0$.
Let $M=(\setC,\clop(\setC))$ be an $L$-\str\ 
of the form $(\setC,h)^{(2)}$ as described in \S\ref{ss:2sorted}, 
 in which the constant $k$ is interpreted as $x$.
Then $\clop(\setC)$ is a base for the topology
on $\setC$, so by lemma~\ref{lem:std tr},
$M\models\{\varphi^x(k/x):\varphi\in\Sigma_0\}$.
Now $\setC$ is a compact \ms, and hence is separable \cite[4.1.18]{Engelking89}.
So  lemma~\ref{lem: C good} applies, and $M$ is good,
giving $M\models T_{good}\cup\{\varphi^x(k/x):\varphi\in\Sigma_0\}$.
Since $\Sigma_0$ was an arbitrary finite subset of $\Sigma$, this
shows that $U$ is consistent and proves the claim.

\medskip

So by \fo\ compactness and the downward L\"owenheim--Skolem theorem,
 we can take a countable model $M=(X,B)\models U$
--- that is, both $X$ and $B$ are countable.
Then $M$ is good, since $M\models T_{good}$.
We apply the preceding work to $M$.
Define $(X,h)=M^{(1)}$, and
 $g:\Var\to\wp(\setC)$ as in definition~\ref{def:cantor asst}.
Since $M\models\varphi^x(k/x)$ for every $\varphi\in\Sigma$,
and lemma~\ref{lem:std tr} applies
to $M$ and $(X,h)$, we obtain
$(X,h),k^M\models\Sigma$.
So by corollary~\ref{cor:extn},   $(\setC,g),\widehat{k^M}\models\Sigma$ as required.
\end{proof}

We can offer no strong completeness result
for $\lbc$ over the Cantor set $\setC$,
because as far as we know,
the $\lbc$-logic of $\setC$ 
has not been determined or axiomatised.
But the logic of $\setC$ in the weaker language $\c L_\bo^{\bdf}$
has been axiomatised by~\cite{Kud06:topdiff}, and this yields:

\begin{corollary}\label{cor:kudinov}
In the language $\c L_\bo^{\bdf}$,
the system \kudsys\
defined in~\cite[\S2]{Kud06:topdiff} is sound and strongly complete over the Cantor set.
\end{corollary}

\begin{proof}
We work in the language $\c L_\bo^{\bdf}$.
As we mentioned in \S\ref{ss:trans},
$\c L_\bo^{\bdf}$ is weaker than $\lbc$,
 so by theorem~\ref{thm:stro compl cantor bo A}
it is compact over $\setC$.
By \cite[lemmas 6 \& 8]{Kud06:topdiff},
\kudsys\ 
is sound, and by \cite[theorem 36]{Kud06:topdiff}, complete, over every
0-\dm al \dims, including of course $\setC$\null.
Strong completeness of \kudsys\ over $\setC$ now follows
by fact~\ref{fact:stro compl = compact}.
\end{proof}

In the still weaker language $\c L^\forall_\bo$, we can present a strong completeness result for all 0-\dm al \dimss:

\begin{corollary}\label{cor:S4U all 0dim}
In the language $\c L^\forall_\bo$,
the system S4U is sound and strongly complete over 
every 0-\dm al \dims.
\end{corollary}

\begin{proof}
By  corollary~\ref{cor:str compl noncomp},
S4U is strongly complete over every non-compact
0-\dm al \dims.
As we mentioned in the proof of the corollary,
S4U is sound and complete over every 
0-\dm al \dims, including the Cantor set.
So by theorem~\ref{thm:stro compl cantor bo A}
and fact~\ref{fact:stro compl = compact},
it is strongly complete over the Cantor set too.
\end{proof}

\section{Conclusion}
We now have some kind of picture of 
compactness and strong completeness 
 over 0-\dm al \dimss\ for 
languages able to express \pa\null.
A summary is in table~\ref{tab1}.
Entries in the two `compact' rows in the table indicate whether the 
designated language is compact (as defined in \S\ref{ss:compactness})
over the relevant space. The {\bf bold} entries imply
the others in the same row.
By fact~\ref{fact:stro compl = compact}, 
a `yes' entry implies strong completeness, 
and a `no' entry implies lack of it,
for any logic named.

\begin{table}[ht]
\begin{center}
\def\myskip{\hskip0.4em\null}
\begin{tabular}{crcccccc}
\hline
$X$&&
\myskip$\bo\forall$\myskip
&
\myskip$\bo[\neq]$\myskip
&
\myskip$\bo\ang n$\myskip
&
\myskip$\bod\forall$\myskip
&
\myskip$\bod\bdf$\myskip
&
\myskip$\bod\ang n$\myskip
\\
\hline
non-compact\myskip&logic:\myskip&S4U&\kudsys&?&KD4U&?$(*)$&?
\\
&compact?\myskip&yes&?&?&\bf yes&?&?
\\[5pt]
Cantor set&logic:\myskip&S4U&\kudsys&?&KD4U&DT$_1$&?
\\
&compact?\myskip&yes&yes&\bf yes&\bf no&no&no
\\[1pt]
\hline
\end{tabular}
\caption{Summary of results for a 0-\dm al \dims\ $X$}\label{tab1}
\end{center}
\end{table}

We now justify some of the statements in the table, make explicit the open questions 
arising from the gaps in the table, and  list some further questions.

Most entries in the table follow from
corollaries~\ref{cor:str compl noncomp},
\ref{cor:kudinov},
and~\ref{cor:S4U all 0dim}, and theorems~\ref{thm:str compl fails cantor [d],A}
and~\ref{thm:stro compl cantor bo A}.
We briefly discuss the penultimate column of the table.
The key fact is:
\begin{fact}[\cite{KudShe14}]\label{fact:KS}
The $\c L^{\bdf}_{\bod}$-logic of
any {separable} 0-\dm al  \dims\ is {\rm DT}$_1$.
\end{fact}
Briefly, {\rm DT}$_1$ can be axiomatised by the KD4 axioms for each of $\bod$ and $[\neq]$,
plus $[\neq]\varphi\to\forall\bod \varphi$.
As we saw in the proof of theorem~\ref{thm:stro compl cantor bo A},
the Cantor set is separable, so
by fact~\ref{fact:KS}, its $\c L^{\bdf}_{\bod}$-logic is {\rm DT}$_1$;
but strong completeness fails,
by theorem~\ref{thm:str compl fails cantor [d],A}.
If fact~\ref{fact:KS} extends to 
non-separable spaces,  the entry marked $(*)$
in the table, currently open, would also be {\rm DT}$_1$.

\begin{problem}\label{prb:general}\rm
Let $X$ be a non-compact 0-\dm al \dims\ (not necessarily separable).
Are the languages 
$\c L^{\bdf}_\bo$, 
$\c L^{\ang n}_\bo$, 
$\c L^{\bdf}_{\bod}$,
and $\c L^{\ang n}_{\bod}$
compact over $X$?
Axiomatise the logic of $X$ in these languages.
(For $\c L^{\bdf}_\bo$ it is \kudsys, as shown by \cite{Kud06:topdiff}.
For $\c L^{\bdf}_{\bod}$ and separable $X$ it is
{\rm DT}$_1$, by fact~\ref{fact:KS}).
Are the logics the same for all  $X$?
\end{problem}

\begin{problem}\label{prb:cantor logic}\rm
Axiomatise the logic of the Cantor set in the languages
$\lbc$ and $\c L^{\ang n}_{\bod}$.
\end{problem}
The language $\c L^{\ang n}_{\bod}$ is important.
 \cite{Gatto:PhD} proved that over T3 spaces
it is  equivalent 
to the monadic 2-sorted \fo\ language $L_t$ of \cite{FZ:top}.
This language can be thought of as
the fragment of the language $L$ of \S\ref{ss:2sorted}
without the boolean function symbols
 that is
invariant under change of base (in $L$-\str s where the set sort is a base
for the topology on the point sort).
 \cite{Gatto:PhD} also gave an axiomatisation of $\c L^{\ang n}_{\bod}$
that is sound and complete
over every class of T1 spaces that contains  all T3 spaces.
This may be relevant to problems~\ref{prb:general} and~\ref{prb:cantor logic}.

0-\dm al spaces are often the easiest to handle.
Going beyond them,
what about arbitrary  \dimss?
Or arbitrary \ms s?
We can ask about the logic of such spaces,
and strong completeness, in 
each of the languages we have considered.
For the  language $\c L_\bo$, the logic
of every \ms\ was determined by \cite{GGL15-ms},
and it seems reasonable to ask about corresponding strong completeness results.
We can even go beyond \ms s and ask for results on non-metrisable topological spaces.
And what about uncountable sets of formulas (when $\Var$ is allowed to be uncountable)?  This is not much explored.
Finally, where compactness fails, can we find 
novel strongly complete deductive
systems using infinitary inference rules?

\bibliographystyle{amsplain}


\providecommand{\bysame}{\leavevmode\hbox to3em{\hrulefill}\thinspace}
\providecommand{\MR}{\relax\ifhmode\unskip\space\fi MR }
\providecommand{\MRhref}[2]{%
  \href{http://www.ams.org/mathscinet-getitem?mr=#1}{#2}
}
\providecommand{\href}[2]{#2}

%
%


\end{document}